\documentclass[12pt,reqno]{amsart}
\usepackage[T1]{fontenc} 
\usepackage{amsmath,amssymb,ifthen}
\usepackage{fullpage,enumerate}
\usepackage{graphicx,psfrag,subfigure}
\usepackage{color}
\usepackage{moreverb}

\def\Tref{T_{\rm ref}}
\def\H{\widetilde{H}}
\def\P{{\mathbb P}}
\def\R{{\mathbb R}}
\def\N{{\mathbb N}}

\def\PP{{\mathcal P}}

\def\SS{{\mathcal S}}
\def\TT{{\mathcal T}}

\def\diam{{\rm diam}}

\def\norm#1#2{\|#1\|_{#2}}
\def\seminorm#1#2{\vert #1\vert_{#2}}
\def\abs#1{\seminorm{#1}{}}

\def\set#1#2{\big\{#1\,:\,#2\big\}}

\def\eps{\varepsilon}
\def\dual#1#2{\langle#1\,,\,#2\rangle}

\def\q{\mathbf{q}}

\def\mesh{\TT}

\def\wat#1{\widehat #1}

\def\slp{\mathfrak{V}} % simple-layer potential
\def\dlp{\mathfrak{K}} % double-layer potential
\def\hyp{\mathfrak{W}} % hypersingular integral operator
\newcounter{constantsnumber}
\def\namec#1#2{%
  \ifthenelse{\equal{#1}{inv}}{C_{\rm inv}}{%
  \ifthenelse{\equal{#1}{invtilde}}{\widetilde C_{\rm inv}}{%
  \ifthenelse{\equal{#1}{lipschitz}}{C_{\rm lip}}{%
  \ifthenelse{\equal{#1}{monotone}}{C_{\rm mon}}{%
  \ifthenelse{\equal{#1}{reliability}}{C_{\rm rel}}{%
  \ifthenelse{\equal{#1}{caccioppoli}}{C_{\rm cacc}}{%
  \ifthenelse{\equal{#1}{xxx}}{\widetilde C_{\rm near}}{%
  \ifthenelse{\equal{#1}{nearfield}}{C_{\rm near}}{%
  \ifthenelse{\equal{#1}{farfield}}{C_{\rm far}}{%
  \ifthenelse{\equal{#1}{cea}}{C_{\mbox{\scriptsize C\'ea}}}{%
  \ifthenelse{\equal{#2}{newcounter}}{\refstepcounter{constantsnumber}\label{const#1}}{}C_{\ref{const#1}}}%
}}}}}}}}}}
\def\setc#1{\namec{#1}{newcounter}}
\def\c#1{\namec{#1}{reference}}

\newtheorem{theorem}{Theorem}[section]
\newtheorem{proposition}[theorem]{Proposition}
\newtheorem{lemma}[theorem]{Lemma}
\newtheorem{corollary}[theorem]{Corollary}

\newtheorem{remark}[theorem]{Remark}
\newtheorem{definition}[theorem]{Definition}
\newtheorem{facts}[theorem]{Facts}

\numberwithin{equation}{section}
\newtheorem{example}[theorem]{Example}

\newcommand{\eremk}{\hbox{}\hfill\rule{0.8ex}{0.8ex}}

\def\subsection#1
{
 \bigskip

 \refstepcounter{subsection}
 {\noindent\bf\arabic{section}.\arabic{subsection}.~#1.~}
}

\def\interior{{\mathrm{int}}}

\def\e{{\rm ext}}
\def\normal{{\boldsymbol\nu}}
\def\mfrac#1#2{\mbox{$\frac{1}{2}$}}
\def\loc{{\ell oc}}

\def\gint{ {\gamma_1^{\rm int}}}
\def\gext{ {\gamma_1^{\rm ext}}}
\def\supp{{\rm supp}}
\def\trace{ {\gamma_0^{\rm int}}}
\def\trext{ {\gamma_0^{\rm ext}}}
  \def\S{\mathbb{S}}

\def\near{{\rm near}}
\def\far{{\rm far}}
\def\el{T}
\def\NN{\mathcal N}
\def\FF{\mathcal F}

\def\wght{w_\index}
\def\P{\mathbb P}

\def\meshsize{h}
\def\index{h}

\begin{document}%%%%%%%%%%%%%%%%%%%%%%%%%%%%%%%%%%%%%%%%%%%%%%%%%%%%%

\title[Inverse Estimates for Boundary Integral Operators]%
{Local Inverse Estimates for\\ Non-Local Boundary Integral Operators}
\date{\today}

\author{M.~Aurada}
\author{M.~Feischl}
\author{T.~F\"uhrer}
\author{M.~Karkulik}
\author{J.\ M.~Melenk}
\author{D.~Praetorius}
\address{Institute for Analysis and Scientific Computing,
       Vienna University of Technology,
       Wiedner Hauptstra\ss{}e 8-10,
       A-1040 Wien, Austria}
       \email{\{Michael.Feischl,\,Melenk,\,Dirk.Praetorius\}@tuwien.ac.at}

\address{Facultad de Matem\'aticas, Pontificia Universidad Cat\'olica de Chile,
Avenida Vicu\~{n}a Mackenna 4860,
Santiago, Chile}
\email{\{tofuhrer,mkarkulik\}@mat.puc.cl}

%%%%%%%%%%%%%%%%%%%%%%%%%%%%%%%%%%%%%%%%%%%%%%%%%%%%%%%%%%%%%%%%%%%%%

\begin{abstract}
We prove local inverse-type estimates for the four non-local
boundary integral  operators associated with the Laplace operator 
on a bounded Lipschitz domain $\Omega$ in $\R^d$ for $d\ge2$ with piecewise smooth boundary.
For piecewise polynomial ansatz spaces and $d \in \{2,3\}$, the inverse
estimates are explicit in both  the local mesh width and the approximation order.
An application to efficiency estimates in {\sl a posteriori} error estimation in boundary element methods is given. 
\end{abstract}

%%%%%%%%%%%%%%%%%%%%%%%%%%%%%%%%%%%%%%%%%%%%%%%%%%%%%%%%%%%%%%%%%%%%%

\maketitle

\section{Introduction}

\noindent
Inverse estimates are general tools for the numerical analysis of discretizations
of partial differential equations (PDEs). They provide a means to bound a stronger
(semi-) norm of a discrete function by a weaker norm up to some negative power of the mesh width. 
For example, in the context of finite element methods, it is textbook knowledge that
\begin{align}\label{intro:fem}
 \norm{\meshsize\nabla V_\index}{L^2(\Omega)} \le C\,\norm{V_\index}{L^2(\Omega)}
 \quad\text{for all continuous $\TT_\index$-piecewise polynomials $V_\index$}.
\end{align}
The constant $C>0$ depends only on the shape regularity of the underlying
triangulation $\TT_\index$ of $\Omega\subset\R^d$ and the polynomial degree of
$V_\index$. Here, $\meshsize\in L^\infty(\Omega)$ is the local mesh width function 
defined by $\meshsize|_T := \diam(T)$ for $T\in\TT_\index$. Inverse
estimates have also been derived for fractional-order Sobolev spaces~\cite{ghs,dfghs}.
The usual proof of inverse estimates like~\eqref{intro:fem} 
relies on scaling arguments, i.e., the powers of $h$ arise by elementwise, i.e., {\em local}
considerations and transformations to reference configurations. 

In the present work we consider the four classical boundary integral operators (BIOs) associated 
with the Laplacian, e.g., the 3D simple-layer integral operator
\begin{align}
 \slp\phi(x) = \frac{1}{4\pi}\,\int_{\partial\Omega} \frac1{|x-y|}\,\phi(y)\,dy
 \quad\text{for }x\in\partial\Omega.
\end{align}
Here, $\Omega\subset\R^d$, $d\geq2$, is a bounded Lipschitz domain 
with piecewise $C^1$-boundary $\partial\Omega$. Let
$\Gamma\subseteq\partial\Omega$ be a relatively open subset of the 
boundary $\partial\Omega$.
Our main result for $\slp$ and $d\in\left\{ 2,3 \right\}$ reads, simplified,
\begin{align}\label{intro:discrete}
  \norm{\meshsize^{1/2}(p+1)^{-1}\nabla_\Gamma \slp\Phi_\index}{L^2(\Gamma)}
 \le C\,\norm{\Phi_\index}{\H^{-1/2}(\Gamma)}
\end{align}
for all $\TT_\index$-piecewise polynomials $\Phi_\index$ of degree $p\in\N_0$,
where $\nabla_\Gamma(\cdot)$ denotes the surface gradient.
The constant $C>0$ depends only on the shape regularity of the underlying 
triangulation $\TT_\index$ of $\Gamma$.
In typical settings, $\slp$ is an isomorphism between $\H^{-1/2}(\Gamma)$ and
$H^{1/2}(\Gamma)$, so that we observe that~\eqref{intro:discrete} is in fact an inverse
estimate for the finite dimensional 
space $\set{\slp\Phi_\index}{\Phi_\index\text{ a }\TT_\index\text{-piecewise polynomial of degree $p\in\N_0$}$} for the weighted 
$H^1$-seminorm and the natural $H^{1/2}$-norm. Inverse estimates of the form~\eqref{intro:discrete} will
be shown for all four BIOs associated with the Laplacian and discrete spaces with spatially varying polynomial degree,
cf. Corollary~\ref{cor:invest}.
In fact, in Theorem~\ref{thm:invest} we will show more general
results of the type
\begin{align}\label{intro:continuous}
  \norm{w_h \nabla_\Gamma \slp\phi}{L^2(\Gamma)}
  \lesssim \left\|\frac{w_h}{h^{1/2}}\right\|_{L^\infty(\Gamma)} \norm{ \phi}{\H^{-1/2}(\Gamma)} + \norm{w_h \phi}{L^{2}(\Gamma)}
\qquad \text{for all } \phi \in L^2(\Gamma), 
\end{align}
where $w_h$ is a fairly general weight function.
The correct choice of the weight function $w_h$ and an inverse estimate from~\cite{ghs,kmr14} 
for the weighted $L^2$-norm allows one to infer \eqref{intro:discrete} from \eqref{intro:continuous}.  

\subsubsection*{Applications}
The inverse-type estimate~\eqref{intro:discrete}
arises naturally in adaptive BEM (boundary element method) when one tries 
to transfer
the convergence and quasi-optimality analysis from adaptive 
FEM~\cite{ckns,stevenson} to adaptive BEM~\cite{fkmp,gantumur}.
Indeed, the present results allow us to prove quasi-optimality of adaptive BEM 
for piecewise smooth geometries
and higher (fixed) order discretizations; we refer to~\cite{part1} and~\cite{part2}, where
this is worked out in detail for weakly singular and hypersingular integral equations, respectively. 
While the inverse estimate~\eqref{intro:discrete} features prominently in the 
analysis of quasi-optimality of adaptive BEM for symmetric problems, it is also a 
key ingredient for plain convergence in non-symmetric problems such as FEM-BEM couplings. 
We refer to~\cite{fembem} and the earlier preprint \cite{invest} of the present 
work for a convergence proof of the adaptive coupling of FEM and BEM. 

A further application of estimate~\eqref{intro:continuous} concerns the
efficiency of weighted residual error estimators for both weakly singular and 
hypersingular integral equations~\cite{cc1997,cms,cmps}. To fix ideas, consider the weakly
singular case and
suppose that $\phi\in L^2(\Gamma)$ solves $\slp\phi=f$ for some given $f\in H^1(\Gamma)$. 
Let $\Phi_h$ be the Galerkin approximation of $\phi$, where the ansatz space consists of 
$\TT_h$-piecewise polynomial of fixed degree $p\in\N_0$. While reliability
\begin{align}
 C_{\rm rel}^{-1}\,\norm{\phi-\Phi_h}{\H^{-1/2}(\Gamma)} \le \eta_{h,\slp} := \norm{h^{1/2}\nabla_\Gamma(f-\slp\Phi_h)}{L^2(\Gamma)}
\end{align}
is well-known (at least for polyhedral domains $\Omega$), the converse efficiency estimate 
remained open. As a consequence of~\eqref{intro:continuous},
we will see in Corollary~\ref{cor:efficient:symm} that 
\begin{align}
 C_{\rm eff}^{-1}\eta_{h,\slp} \le \norm{h^{1/2}(\phi-\Phi_h)}{L^2(\Gamma)}, 
\end{align}
which expresses efficiency of the weighted residual error estimator with respect to 
the slightly stronger norm 
$\norm{h^{1/2}(\phi-\Phi_h)}{L^2(\Gamma)} \gtrsim \norm{\phi-\Phi_h}{\H^{-1/2}(\Gamma)}$. 
We refer to Corollary~\ref{cor:efficient:hypsing} for the case of the hypersingular operator.

These efficiency bounds are specific instances of new stability estimates for the BIOs
in locally weighted $L^2$-norms detailed in Corollaries~\ref{cor:stabV} and~\ref{cor:stabW}.

\subsubsection*{Novelty}
The discrete inequality~\eqref{intro:discrete} was first shown independently in 
\cite{fkmp} and \cite{gantumur}, however, under some restrictions. The work~\cite{fkmp} considers only
lowest-order polynomials, i.e., $\TT_\index$-piecewise constants, but works for
polyhedral boundaries $\Gamma$. The work~\cite{gantumur} 
proves~\eqref{intro:discrete} for arbitrary $\TT_\index$-piecewise polynomials, but
its wavelet-based analysis is restricted to $C^{1,1}$-boundaries 
$\Gamma$ and the constant $C>0$ depends on the polynomial degree.
Our proof of~\eqref{intro:continuous} generalizes the works \cite{fkmp,gantumur} 
in the following ways: 1) we generalize the analysis of \cite{fkmp} for the 
simple-layer operator $\slp$ to all four BIOs associated
with the Laplacian (i.e., the double-layer operator $\dlp$, its adjoint $\dlp'$, 
and the hypersingular operator $\hyp$); 2) we extend our previous analysis from 
polyhedral domains to piecewise smooth geometries; 3) we lift the restriction 
to fixed-order polynomial ansatz space and permit very general ansatz spaces; 
4) for ansatz spaces of piecewise polynomials of arbitrary order, we make the 
dependence on the  polynomial degree in the inverse estimates explicit. 

The technical difficulty in the proof of~\eqref{intro:continuous} and \eqref{intro:discrete}
lies in the non-locality of the boundary integral operator $\slp$, which precludes simple elementwise
considerations. 
We cope with the non-locality of the BIOs
by splitting them
into near-field and far-field contributions, each requiring different tools. 
The analysis of the near-field part relies on local arguments and stability 
properties of the BIOs. For the far-field part, the key observation
is that the BIOs are derived from two volume potentials, namely, the 
simple-layer potential $\widetilde \slp$ and the double-layer potential 
$\widetilde \dlp$ by taking appropriate traces. Since these potentials solve 
elliptic equations, ``interior regularity'' estimates 
are available for them and trace inequalities imply corresponding estimates 
for the BIOs. 
Section~\ref{section:invest:aux} proves the
relevant estimates for the simple-layer potential $\widetilde\slp$,
whereas Section~\ref{sec:aux-dlp} is concerned with the 
double-layer potential $\widetilde\dlp$. 
The final Section~\ref{section:proof} then combines these results to give the proof of Theorem~\ref{thm:invest}.

Although the present paper considers only the four BIOs associated with the Laplacian,
the scope is wider. As just mentioned, the key tool are interior estimates for potentials; such 
estimates are available for many elliptic equations, for example, the Lam\'e system, 
so that we expect that corresponding results can be proved as well for BIOs associated with 
these problems.

\subsubsection*{General notation}
We close the introduction by stating that $\abs{\cdot}$ denotes, depending on the context, 
the absolute value of a real number, the Euclidean norm of a vector in $\R^d$, the Lebesgue measure of
a subset of $\R^{d-1}$ or $\R^d$ 
or the $(d-1)$-dimensional surface measure of a subset of $\partial\Omega$.
The notation $a\lesssim b$ abbreviates $a\leq C\cdot b$ for some constant $C>0$ which
will be clear from the context, and we write $a\simeq b$ to abbreviate
$a\lesssim b \lesssim a$.

We write $B_r(x):=\set{ z\in\R^d}{ |x-z|\leq r}$ for the closed ball with 
radius $r$ and center $x$.

%%%%%%%%%%%%%%%%%%%%%%%%%%%%%%%%%%%%%%%%%%%%%%%%%%%%%%%%%%%%%%%%%%%%%%%%%%%%%%%

\section{Spaces, Operators, and Meshes}
\label{section:preliminaries}

%-------------------------------------------------------------------------
\subsection{Sobolev spaces}
\label{section:invest:sobolev}
%-------------------------------------------------------------------------
$\Omega$ is a bounded Lipschitz domain in $\R^d$, $d\ge2$, 
with piecewise $C^1$-boundary $\partial\Omega$ and corresponding exterior domain
$\Omega^\e:=\R^d\setminus\overline\Omega$. The exterior unit normal vector 
field on $\partial\Omega$ is denoted by $\normal$.  Throughout, we will assume that 
$\Gamma\subseteq\partial\Omega$ is a non-empty, relatively open set
  that stems from a Lipschitz dissection 
  $\partial\Omega = \Gamma \cup \partial\Gamma \cup (\partial\Omega\setminus\Gamma)$
as described in~\cite[pp.~99]{mclean}. Note that $\Gamma=\partial\Omega$ is valid.

The non-negative order  Sobolev spaces $H^{1/2+s}(\partial\Omega)$ for $s \in \{-1/2,0,1/2\}$ are 
defined as in \cite[pp.~{99}]{mclean} by use of Bessel potentials on $\R^{d-1}$ and lifting via the bi-Lipschitz maps that describe $\partial\Omega$. We also need the spaces $H^{1/2+s}(\Gamma)$ and 
$\H^{1/2+s}(\Gamma)$. In accordance with \cite{mclean}, these are defined as follows: 
\begin{align}
\label{eq:H1/2+s}
H^{1/2+s}(\Gamma) &:= \{v|_\Gamma \colon  v \in H^{1/2+s}(\partial\Omega)\}, \\
\label{eq:tildeH1/2+s}
\H^{1/2+s}(\Gamma) &:= \{v \colon  E_{0,\Gamma} v \in H^{1/2+s}(\partial\Omega)\}, 
\end{align}
where $E_{0,\Gamma}$ denotes the operator that extends a function defined on $\Gamma$ to a function 
on $\partial\Omega$ by zero. These spaces are endowed with their natural norms, i.e., the quotient norm 
$\|v\|_{H^{1/2+s}(\Gamma)}:= \inf \{ \|V\|_{H^{1/2+s}(\partial\Omega)}\colon V|_\Gamma = v\}$ and 
$\|v\|_{\H^{1/2+s}(\Gamma)} := \|E_{0,\Gamma} v\|_{H^{1/2+s}(\partial\Omega)}$. 
Owing to the assumption that $\partial\Omega = \Gamma \cup \partial\Gamma \cup (\partial\Omega\setminus \Gamma)$ 
is a Lipschitz dissection, we have the following facts, stated here without proof: 
\begin{facts}
\label{facts:sobolev}
\begin{enumerate}[(i)]
\item 
\label{item:facts:sobolev-i}
For $s = 1/2$, we have the norm equivalences 
$\|u\|^2_{H^1(\partial\Omega)} \simeq \|u\|^2_{L^2(\partial\Omega)} + \|\nabla_\Gamma u\|^2_{L^2(\partial\Omega)}$
and 
$\|u\|^2_{\H^1(\Gamma)} \simeq \|u\|^2_{L^2(\Gamma)} + \|\nabla_\Gamma u\|^2_{L^2(\Gamma)}$, 
where $\nabla_\Gamma$ is the (weak) surface gradient. 
\item 
\label{item:facts:sobolev-ii}
For $s = 0$, the norms $\|u\|_{H^{1/2}(\partial\Omega)}$ and $\|u\|_{\H^{1/2}(\Gamma)}$ can equivalently
be described by the Aronstein-Slobodeckii norms of $u$ and $E_{0,\Gamma} u$ 
(cf. \cite[(3.18)]{mclean} for the definition of the Aronstein-Slobodeckii norm).
\item 
\label{item:facts:sobolev-iii}
For $s = 0$, the spaces $H^{1/2}(\partial\Omega)$ and $\H^{1/2}(\Gamma)$ are obtained from 
interpolating between the cases $s = -1/2$ (i.e., $L^2(\partial\Omega)$ or $L^2(\Gamma)$) and 
$s = 1/2$ (i.e., $H^1(\partial\Omega)$ or $\H^1(\Gamma)$) using the $K$-method (cf., e.g., \cite[{Thm.~B.11}]{mclean} for the 
case of $H^s(\partial\Omega)$).  
\end{enumerate}  
\end{facts}
Negative order Sobolev spaces are defined by duality, namely, for $s \in \{-1/2,0,1/2\}$, 
\begin{align}\label{eq:negorder}
\begin{split}
  H^{-1/2}(\partial\Omega) &:= H^{1/2}(\partial\Omega)',\\
  \H^{-(1/2+s)}(\Gamma) &:= H^{1/2+s}(\Gamma)',\quad\text{and}\\
  H^{-(1/2+s)}(\Gamma) &:= \H^{1/2+s}(\Gamma)',
  \end{split}
\end{align}
where duality pairings $\langle \cdot ,\cdot\rangle$ are understood to extend the standard $L^2$-scalar product on $\partial\Omega$ or $\Gamma$.
We observe the continuous inclusions
\begin{align*}
  \H^{\pm(1/2+s)}(\Gamma) &\subseteq H^{\pm(1/2+s)}(\Gamma) 
 \quad\text{as well as}\quad %\gamma\subset\Gamma,\\
  \H^{\pm(1/2+s)}(\partial\Omega) = H^{\pm(1/2+s)}(\partial\Omega). 
\end{align*}

We also note that for $\psi \in L^2(\Gamma)$ the zero extension $E_{0,\Gamma} \psi$ satisfies $E_{0,\Gamma} \psi \in H^{-1/2}(\partial\Omega)$ with 
\begin{equation}
\label{eq:normeqiv}
\|\psi\|_{\H^{-1/2}(\Gamma)} = \|E_{0,\Gamma} \psi\|_{H^{-1/2}(\partial\Omega)}. 
\end{equation}

We denote by $\trace(\cdot)$ the interior trace operator, i.e., $\trace u$ is the 
restriction of a function $u\in H^1(\Omega)$ to the boundary $\partial\Omega$. 
With $\displaystyle   H^1_\Delta(\Omega):=
  \set{u \in H^1(\Omega)}{ -\Delta u \in L^2(\Omega)}
$, the interior conormal derivative operator 
$\gamma_1^\interior:H^1_\Delta(\Omega)\rightarrow H^{-1/2}(\partial\Omega)$
is defined by the 
first Green's formula, viz., 
\begin{align}
  \label{eq:first Green identity}
  \dual{\gamma_1^\interior u}{v}_{\partial\Omega}
  =\dual{\nabla u}{\nabla v}_\Omega
  - \dual{-\Delta u}{v}_\Omega
  \quad\text{for all }v\in H^1(\Omega).
\end{align}

\begin{remark}
\label{rem:gamma_1}
The operator $\gamma_1^\interior(\cdot)$ generalizes the classical normal derivative 
operator: if $u \in H^1_\Delta(\Omega)$ is sufficiently smooth near a boundary
point $x_0$, then  $\gamma_1^\interior u$ can be represented near $x_0$ by a function 
given by the pointwise defined  normal derivative $\partial_{\normal} u$.  
\eremk
\end{remark}

The exterior trace $\trext$ and the exterior conormal derivative
operator $\gext$ are defined analogously to their interior counterparts. 
To that end, we fix a bounded Lipschitz domain $U \subset {\mathbb R}^d$ 
with $\overline{\Omega} \subset U$. The exterior trace operator 
$\trext:H^1(U\setminus\overline{\Omega}) \rightarrow H^{1/2}(\partial\Omega)$ 
is defined by restricting to $\partial\Omega$, and the exterior conormal derivative 
$\gext$ is characterized by 
 $ \dual{\gext u}{v}_{\partial\Omega}
  =\dual{\nabla u}{\nabla v}_{U\setminus\overline{\Omega}}
  - \dual{-\Delta u}{v}_{U\setminus\overline{\Omega}}$
  for all $v\in H^1(U\setminus\overline{\Omega})$ with $\trext v = 0$ on $\partial U$.

For a function $u$ that admits both conormal derivatives or both traces, we define
the jumps
$[\gamma_1 u] := \gext u - \gint u$ and  $[u] = \trext u - \trace u$, respectively.

%-------------------------------------------------------------------------
\subsection{Boundary integral operators}
\label{section:bios}
%-------------------------------------------------------------------------
We briefly introduce the pertinent boundary integral operators and refer to 
the monographs~\cite{mclean,hsiaowendland,ss} for further details and proofs. 
%-------------------------------------------------------------------------
Green's function for the Laplace operator is given by 
\begin{align}
 G(x,y) = \begin{cases}
 -\frac{1}{|\S^1|}\,\log|x-y|,\quad&\text{for }d=2,\\
 +\frac{1}{|\S^{d-1}|}\,|x-y|^{-(d-2)},&\text{for }d\ge3,
 \end{cases}
\end{align}
where $|\S^{d-1}|$ denotes the surface measure of the Euclidean sphere in $\R^d$, e.g., $|\S^1|=2\pi$ and $|\S^2|=4\pi$.
The classical simple-layer potential $\widetilde\slp$ and the double-layer potential
$\widetilde \dlp$ are formally defined by
\begin{equation*}
(\widetilde \slp \psi)(x):= \int_{\partial\Omega} G(x,y) \psi(y)\,dy, 
\qquad 
(\widetilde \dlp v)(x):= \int_{\partial\Omega} \partial_{\normal(y)} G(x,y) v(y)\,dy, 
\qquad x \in {\mathbb R}^d\setminus \partial\Omega; 
\end{equation*}
here, $\partial_{\normal(y)}$ denotes the (outer) normal derivative with respect to the variable $y$. 
These pointwise defined operators can be extended to bounded linear
operators
\begin{align}
\label{eq:mapping-properties-potentials}
 \widetilde\slp\in L\big(H^{-1/2}(\partial\Omega);H^1(U)\big)
 \quad\text{and}\quad
 \widetilde\dlp\in L\big(H^{1/2}(\partial\Omega);H^1(U\setminus\partial\Omega)\big).
\end{align}
It is well-known that 
$\Delta\widetilde\slp\psi = 0 = \Delta\widetilde\dlp v$ in $U\setminus\partial\Omega$ 
for all 
$\psi\in H^{-1/2}(\partial\Omega)$ and $v\in
H^{1/2}(\partial\Omega)$.
The simple-layer, double-layer, adjoint double-layer, and the hypersingular integral operator are 
defined as follows: 
\begin{align}
  \slp = \trace\widetilde\slp, \quad \dlp = \frac{1}{2}+\trace\widetilde\dlp, \quad \dlp'=-\frac{1}{2}+\gint\widetilde\slp, \text{ and } \quad \hyp = -\gint\widetilde\dlp.
\end{align}
These linear operators are bounded linear operators for $s \in \{-1/2,0,1/2\}$ as follows: 

\begin{align}
 \label{def:slp}
 \slp&\in L(H^{-1/2+s}(\partial\Omega);H^{1/2+s}(\partial\Omega)),
% &(\slp\psi)(x) &= \int_{\partial\Omega} G(x,y)\,\psi(y)\,dy,\\
 \\
 \label{def:dlp}
 \dlp&\in L(H^{1/2+s}(\partial\Omega);H^{1/2+s}(\partial\Omega)),
% &(\dlp v)(x) &= \int_{\partial\Omega} \partial_{\normal(y)}G(x,y)\,v(y)\,dy,\\
 \\
 \label{def:adlp}
 \dlp'&\in L(H^{-1/2+s}(\partial\Omega);H^{-1/2+s}(\partial\Omega)),
% &(\dlp'\psi)(x) &= \int_{\partial\Omega} \partial_{\normal(x)}G(x,y)\,\psi(y)\,dy,\\
 \\
 \label{def:hyp}
 \hyp&\in L(H^{1/2+s}(\partial\Omega);H^{-1/2+s}(\partial\Omega)),
% &(\hyp v)(x) &= -\partial_{\normal(x)}\int_{\partial\Omega} \partial_{\normal(y)}G(x,y)\,v(y)\,dy.
\end{align}

The operators $\widetilde \slp$, $\slp$, $\dlp'$ will often be applied to functions in $L^2(\Gamma)$. 
Throughout the paper, we employ the convention that for $\psi \in L^2(\Gamma)$ we implicitly extend by zero, e.g., 
\begin{equation}
\label{eq:convention}
\mbox{$\widetilde\slp \psi$ means $\widetilde \slp (E_{0,\Gamma} \psi$),} 
\quad 
\mbox{$\slp \psi$ means $\slp (E_{0,\Gamma} \psi$),} 
\quad 
\mbox{and $\dlp' \psi$ means $\dlp' (E_{0,\Gamma} \psi$).} 
\end{equation}
An analogous extension is obviously used when $\widetilde \dlp$, $\dlp$, $\hyp$ are applied to an 
$v \in \H^{1/2}(\Gamma)$. 
\begin{remark}
Ellipticity of $\slp$ and $\hyp$ is not used in our analysis of Theorem~\ref{thm:invest}
and Corollary~\ref{cor:invest}.
In particular, there is no need to scale $\Omega$ to ensure 
$\diam(\Omega)<1$ in 2D or to assume that $\Gamma$ is connected. 
\eremk
\end{remark}

%-------------------------------------------------------------------------
\subsection{Surface simplices and admissible triangulations}\label{section:surfsimp}
%-------------------------------------------------------------------------
Fix the reference simplex 
$T_{\rm ref} := \{ x\in\R^{d-1}, \, 0 < x_1, \dots, x_{d-1}, \sum_{j=1}^{d-1}x_j < 1\}$, which is 
the convex hull of the $d$ vertices $\{0,e_1,\ldots,e_{d-1}\}$ (``$0$-faces''). The convex hull of any $j+1$ of these vertices is called
a ``$j$-face'' of $\Tref$. We call the $(d-2)$-faces ``facets'' of $\Tref$. 

We require the concept of regular, shape-regular triangulations ${\mesh_\index}$ of $\Gamma$. 
\begin{definition}[regular and shape-regular triangulations] 
\label{def:mesh}
A set $\mesh_\index$ of subsets of $\Gamma$ is called a \emph{regular} triangulation of $\Gamma$ if 
the following is true: 
\begin{enumerate}[(i)]
\item 
\label{item:def:mesh-i}
The {\em elements} $T \in \mesh_\index$ are relatively open subsets of $\Gamma$ and each $T$ is the image of $\Tref$
under an {\em element map} $\gamma_T:\overline{\Tref} \rightarrow \overline{T}$. The element map
$\gamma_T$ is assumed to be bijective and $C^1$ on $\overline{\Tref}$. 
\item 
\label{item:def:mesh-ii}
The elements cover $\Gamma$: $\bigcup_{T \in \mesh_\index} \overline{T}  = \overline{\Gamma}$. 
\item 
\label{item:def:mesh-iii}

``no hanging nodes'': 
For each pair $(T,T^\prime) \in \mesh_\index\times \mesh_\index$, the intersection $\overline{T} \cap \overline{T^\prime}$ 
is either empty or there are two $j$-faces $f$, $f^\prime \subseteq \partial\Tref$ of $\Tref$ with 
$j \in \{0,\ldots,d-2\}$ such that 
$\overline{T} \cap \overline{T^\prime}  = \gamma_T(f) = \gamma_{T^\prime}(f^\prime)$. 
\item 
\label{item:def:mesh-iv}
Parametrizations of common boundary parts of neighboring elements are compatible: 
If $\emptyset \ne \overline{T} \cap \overline{T^\prime}  = \gamma_T(f) = \gamma_{T^\prime}(f^\prime)$, then 
$\gamma_{T}^{-1} \circ \gamma_{T^\prime}:f^\prime \rightarrow f$ is an affine isomorphism. 
\end{enumerate}
We call the images of vertices of $\Tref$ under the element maps {\em nodes} of $\mesh_\index$ 
and collect them in the set $\NN_\index$.  The 
images of the $(d-2)$-faces of $\Tref$ are called {\em facets} of $\mesh_\index$
and collected in the set $\FF_\index$.
For each $T \in \mesh_\index$, we set 
$h(T):= \operatorname*{diam}(T) := \sup_{x,y \in T} |x - y|$. 

A regular triangulation is called \emph{$\kappa$-shape regular}, if the element maps $\gamma_T$ satisfy the following:
\begin{enumerate}[(i)]
\setcounter{enumi}{4}
\item 
\label{item:def:mesh-v}
Let $G_T(x):= \gamma_T'(x)^\top\cdot\gamma_T'(x)\in \R^{(d-1)\times(d-1)}$ be the symmetric Gramian matrix of $\gamma_T$. The triangulation is $\kappa$-shape regular if for 
all $T \in \mesh_\index$ the extremal eigenvalues $\lambda_{min} (G_T(x))$ and $\lambda_{max} (G_T(x))$ 
of $G_T(x)$ satisfy 
$$
\sup_{x  \in \Tref} \left( \frac{h(T)^{2}}{\lambda_{min}(G_T(x))} + \frac{\lambda_{max}(G_T(x))}{h(T)^{2}}\right) \leq \kappa. 
$$
\item
\label{item:def:mesh-vi}
If $d = 2$, we require explicitly that the element sizes of neighboring elements are comparable: 
$$
h(T) \leq \kappa h(T^\prime) \qquad \text{for all }T, T^\prime \mbox{ with } \overline{T} \cap \overline{T^\prime} \ne \emptyset.
$$
\end{enumerate}
\end{definition}
With each triangulation $\mesh_\index$, we associate the local mesh size function $\meshsize\in L^\infty(\Gamma)$
which is defined elementwise by $\meshsize|_T := \meshsize(T)$ for all $T\in\mesh_\index$.
We note that for a $\kappa$-shape regular mesh we have 
\begin{align}\label{def:kappa:3d}
 \max_{T\in\mesh_\index}\frac{h(T)^{d-1}}{|T|} \lesssim 1, 
\end{align}
where the implied constant depends solely on $\kappa$. 

If $\Gamma$ is the union of pieces of $(d-1)$-dimensional hyperplanes and the element 
maps are {\em affine}, then the Gramians are constants and the Definition~\ref{def:mesh} generalizes
the classical concept of a shape-regular triangulation of $\Gamma$. In the non-affine case, the following
example illustrates how triangulations as stipulated in Definition~\ref{def:mesh} can be created: 
\begin{example}
Let $\Gamma \subseteq \partial\Omega$ be an open surface piece and assume 
$\Gamma = \gamma(\widehat\Gamma)$ for some reference configuration $\widehat\Gamma \subseteq \R^{d-1}$
and some sufficiently smooth map $\gamma$. 
Let $\widehat \mesh_\index = \{\widehat T_1,\ldots,\widehat T_N\}$ be a standard, regular, shape-regular 
triangulation of $\widehat\Gamma$ with affine element maps $\widehat\gamma_{\widehat T_i}$, $i=1,\ldots,N$. 
Then, the triangulation with elements $T = \gamma \circ \gamma_{\widehat T_i} (\Tref)$ and 
element maps $\gamma \circ \gamma_{\widehat T_i}$ satisfies the hypotheses of Definition~\ref{def:mesh}.
\newline 
This concept generalizes to surfaces consisting of several patches; it is worth emphasizing that in that 
case the patch parametrizations need to match at patch boundaries.
\eremk
\end{example}
\begin{remark}
In Definition~\ref{def:mesh} 
the conditions on the mesh are formulated so as to ensure that the spaces $\SS^{\q+1}(\mesh_\index)$
below have good approximation properties. The conditions (\ref{item:def:mesh-iii}) and (\ref{item:def:mesh-iv}) 
in Definition~\ref{def:mesh} could be relaxed if only good approximation properties of the spaces 
$\PP^\q(\mesh_\index)$ are required. 
\eremk
\end{remark}

For an element $T \in \TT_\index$, we define 
the element patch $\omega_\index(T)$ by 
\begin{align}\label{def:patch}
 \omega_\index(T) := \left( \bigcup\set {\overline{T'}}{T^\prime \in \TT_\index \mbox{ with } \overline{T}\cap \overline{T'}\neq\emptyset}\right)^\circ. 
\end{align}%
The assumptions on the element maps of a $\kappa$-shape regular
triangulation imply that elements of a patch are comparable in size. Furthermore, 
the fact that $\Gamma$ results from a Lipschitz dissection of $\partial\Omega$ imposes certain topological
restrictions on the patches: 
\begin{lemma}
\label{lemma:patch}
Let $\mesh_\index$ be a regular, $\kappa$-shape regular mesh. Then there is a constant $C > 0$ that depends 
solely on $\kappa$ and the Lipschitz character of $\partial\Omega$ such that the following holds:
\begin{enumerate}[(i)]
\item
\label{item:lemma:patch-i}
$h(T) \leq C h(T^\prime)$ 
for any two elements $T$, $T^\prime$ with $\overline{T} \cap \overline{T^\prime} \ne \emptyset$. 
\item 
\label{item:lemma:patch-ii}
The number of elements in an element patch is bounded by $C$. 
\item 
\label{item:lemma:patch-iii}
For any two elements $T$, $T^\prime$
in the element patch $\omega_\index( T^{\prime\prime})$ there is a sequence $T = T_0, \ldots, T_n = T^\prime$ 
of elements $T_i$, $i=0,\ldots,n$, in $\omega_\index(T^{\prime\prime})$ such that two successive elements $T_i$, $T_{i+1}$ 
share a common facet: $\overline{T_i} \cap \overline{T_{i+1}} \in \FF_\index$ for $i=0,\ldots,n-1$. 
\end{enumerate}
\end{lemma}
\begin{proof}[Sketch of Proof]
{\em Statement~(\ref{item:lemma:patch-iii}):}
We show~\eqref{item:lemma:patch-iii} first for the node patch 
\begin{align*}
\omega_\index(z):=  \left( \bigcup\set {\overline{T}}{T \in \TT_\index \mbox{ with } z\in\overline{T}}\right)^\circ
\end{align*}
and some node $z$ of $T^{\prime\prime}$.
This follows from the fact that $\Gamma$ results from a Lipschitz dissection
and considerations in $\R^{d-1}$ using local charts. After a Euclidean change of coordinates, we may 
assume that $\partial\Omega$ is (locally) a hypograph, i.e., there is a Lipschitz continuous function 
$\Lambda: B_r(0) \rightarrow \R$ such that the set $\{(x,\Lambda(x))\colon x \in B_r(0)\} \subset \partial\Omega$. 
Without loss of generality, we assume the Euclidean coordinate change is such that $z = (0,\Lambda(0))$. 
One may also assume that $\Lambda$ is defined on $\R^{d-1}$ (and Lipschitz continuous) so that the 
map $\widetilde\Lambda:\R^d \rightarrow \R^d$ given by $(x,t) \mapsto (x,\Lambda(x)+t)$ is bilipschitz. 
\newline 
We distinguish the cases 
$z \in \Gamma$ and $z \in \partial\Gamma$. Let $z$ be an interior point of $\Gamma$. Then, 
the pull-backs $\widehat T:= \widetilde \Lambda^{-1}(T)$, $T \subseteq \omega_\index(z)$, 
are contained in the hyperplane $\R^{d-1} \times \{0\}$ and (identifying this hyperplane with $\R^{d-1}$) 
completely cover a neighborhood of 
$0 \in \R^{d-1}$. This together with (\ref{item:def:mesh-iii}) of Definition~\ref{def:mesh} shows the claim. 
If $z \in \partial\Gamma$, then the fact 
that the elements are contained in $\Gamma$
and that $\Gamma$ results from a Lipschitz dissection implies that near $0\in\R^{d-1}$, 
the pull-backs $\widehat T$ are all on one side of a Lipschitz graph in $\R^{d-1}$. This together 
with (\ref{item:def:mesh-iii}) of Definition~\ref{def:mesh} again implies the claim. Since $\omega_\index(T^{\prime\prime})$ is the union of the $d$ node patches $\omega_\index(z)$ associated with the $d$ nodes of $T^{\prime\prime}$, this concludes the proof of~\eqref{item:lemma:patch-iii}.

{\em Statement~(\ref{item:lemma:patch-ii}):} 
Consider the case of an interior point $z \in \Gamma$. The assumption 
(\ref{item:def:mesh-iii}) of Definition~\ref{def:mesh} and the fact that the map $\widetilde \Lambda$
is bilipschitz implies that the solid angles of the elements $\widehat T$ at $0$ are bounded away from 
zero by a constant that depends solely on $\kappa$ and $\widetilde \Lambda$. This implies the claim for 
a node patch $\omega_\index(z)$ and thus for $\omega_\index(T)$ with $T\in\mesh_\index$. 
\newline 
{\em Statement~(\ref{item:lemma:patch-i}):} 
For $d = 2$, this follows by definition. For $d \ge 3$ we first note that two 
elements sharing a facet $f \in \FF_\index$ have comparable size by 
(\ref{item:def:mesh-iii})---(\ref{item:def:mesh-v}) of Definition~\ref{def:mesh}. 
We conclude the proof with the aid of statements (\ref{item:lemma:patch-iii}) and (\ref{item:lemma:patch-ii}). 
\end{proof}
%-------------------------------------------------------------------------
\subsection{Admissible weight functions and discrete spaces}
\label{section:bem}
%-------------------------------------------------------------------------
\begin{definition}[$\sigma$-admissible weight functions and polynomial degree distributions]
A function $w_\index\in L^\infty(\Gamma)$ 
is $\sigma$-{\em admissible} with respect to $\TT_\index$ if 
\begin{align*}
 \norm{w_\index}{L^\infty(T)} \leq \sigma \, w_\index(x) 
 \quad\text{almost everywhere on } \omega_\index(T).
\end{align*} 
A $\sigma$-admissible function $q_\index\in L^\infty(\Gamma)$ is called a 
\emph{ $\sigma$-admissible polynomial degree distribution} 
with respect to $\TT_\index$, if $q_\index(T):=q_\index|_T\in\N_0$ for all 
$T\in\TT_\index$.
\end{definition}
We write 
\begin{align}
\PP^{\q}(\mesh_\index)
 := \set{\Psi_\index\in L^2(\Gamma)}{\forall T\in\mesh_\index\quad
 \Psi_\index\circ\gamma_T\text{ is a polynomial of degree }\le q_\index(T)},
\end{align}
for the space of (discontinuous) piecewise polynomials of local degree
$q_\index(T)$. Moreover, we introduce spaces of continuous piecewise polynomials of local degree $q_\index(T)+1$ by 
\begin{align}
\SS^{\q+1}(\mesh_\index) &:= \PP^{\q+1}(\mesh_\index)\cap H^1(\Gamma), \\
 %\widetilde\SS^{\q+1}(\mesh_\index) &:= \set{ V_\index|_\Gamma}{V_\index\in\SS^{\q+1}(\mesh_\index) \text{ with } \supp(V_\index)\subseteq\overline\Gamma}. 
 \widetilde\SS^{\q+1}(\mesh_\index) &:= \SS^{\q+1}(\mesh_\index) \cap \H^1(\Gamma). 
\end{align}
We note the inclusions
$\PP^{\q}(\mesh_\index)\subset L^2(\Gamma) \subset \H^{-1/2}(\Gamma)$, 
$\widetilde\SS^{\q+1}(\mesh_\index)\subset \H^1(\Gamma) \subset \H^{1/2}(\Gamma)$, 
and $\SS^{\q+1}(\mesh_\index)\subset H^1(\Gamma)$, as well as 
$\widetilde\SS^{\q+1}(\mesh_\index) = \SS^{\q+1}(\mesh_\index)$ in case of 
$\Gamma=\partial\Omega$. 

For $q \in \N_0$, the use of non-boldface superscripts in $\PP^q(\mesh_\index)$, $\SS^{q+1}(\mesh_\index)$, and 
$\widetilde \SS^{q+1}(\mesh_\index)$ indicates that a constant polynomial degree is employed. 
%
%%%%%%%%%%%%%%%%%%%%%%%%%%%%%%%%%%%%%%%%%%%%%%%%%%%%%%%%%%%%%%%%%%%%%%%%%%%%%%%

\section{Main result and applications}
\label{section:statement}
%------------
\subsection{Inverse estimates}
%------------
\noindent
The following Theorem~\ref{thm:invest} is the main result of this work. 

\begin{theorem}
\label{thm:invest}
Let $\mesh_\index$ be a regular, $\kappa$-shape regular triangulation of $\Gamma$ and let $\wght\in L^\infty(\Gamma)$ be a $\sigma$-admissible weight function with respect to $\TT_\index$.
Then, it holds
\begin{align}
\label{eq:invest:V}
\norm{\wght\nabla_\Gamma\slp\psi}{L^2(\Gamma)}
+ \norm{\wght\dlp'\psi}{L^2(\Gamma)}
& \le\c{inv}\big(\norm{\wght/\meshsize^{1/2}}{L^\infty(\Gamma)}\norm{\psi}{\H^{-1/2}(\Gamma)}
+ \norm{\wght\psi}{L^2(\Gamma)}\big),
\\
\label{eq:invest:K}
\norm{\wght\nabla_\Gamma\dlp v}{L^2(\Gamma)}
+\norm{\wght\hyp v}{L^2(\Gamma)}
&
\le\c{inv}\big(\norm{\wght/\meshsize^{1/2}}{L^\infty(\Gamma)}\norm{v}{\H^{1/2}(\Gamma)}
+ \norm{\wght\nabla_\Gamma v}{L^2(\Gamma)}\big),
\end{align}
for all functions $\psi \in L^2(\Gamma)$
and all $v \in \widetilde H^1(\Gamma)$.
The constant $\c{inv}>0$ depends only on $\partial\Omega$, $\Gamma$, the
$\kappa$-shape regularity of $\mesh_\index$, and $\sigma$.
\end{theorem}
In the following Corollary~\ref{cor:invest}, we apply the 
estimates~\eqref{eq:invest:V}--\eqref{eq:invest:K} of Theorem~\ref{thm:invest} 
to discrete functions $\Psi_\index\in\PP^\q(\mesh_\index)$ and $V_\index\in\widetilde\SS^{\q+1}(\mesh_\index)$.
We mention that the restriction
to $d \in \{2,3\}$ in Corollary~\ref{cor:invest} is due to the fact that the underlying reference
\cite{kmr14} restricts to this setting. 
\begin{corollary}
\label{cor:invest}
Let $\mesh_\index$ be a regular, $\kappa$-shape regular triangulation of $\Gamma$.
Suppose that $d\in\left\{ 2,3 \right\}$ and
that $q_\index$ is a $\sigma$-admissible polynomial degree 
distribution with respect to $\TT_\index$. %Let $q_\index+1$ be defined pointwise. 
Then, there exists a constant $\setc{invtilde}>0$ such that the following
estimates hold:
  \begin{align}
   \label{eq:cor:invest:V}
   \norm{\meshsize^{1/2}(q_\index+1)^{-1}%\max\{1,q_\index\}^{-1}
   \,\nabla_\Gamma\slp\Psi_\index}{L^2(\Gamma)}
+    \norm{\meshsize^{1/2}(q_\index+1)^{-1}%\max\{1,q_\index\}^{-1}
\,\dlp'\Psi_\index}{L^2(\Gamma)}
   &\le\c{invtilde}\norm{\Psi_\index}{\H^{-1/2}(\Gamma)},\\
   \label{eq:cor:invest:K}
   \norm{\meshsize^{1/2}(q_\index+1)^{-1}\,\nabla_\Gamma\dlp V_\index}{L^2(\Gamma)}
+   \norm{\meshsize^{1/2}(q_\index+1)^{-1}\,\hyp V_\index}{L^2(\Gamma)}
   &\le\c{invtilde}\norm{V_\index}{\H^{1/2}(\Gamma)},
  \end{align}
  for all discrete functions $\Psi_\index\in\PP^{\q}(\mesh_\index)$ and
  $V_\index\in\widetilde\SS^{\q+1}(\mesh_\index)$.
  The constant $\c{invtilde}>0$ depends only on $\partial\Omega$, $\Gamma$, the
  $\kappa$-shape regularity of $\mesh_\index$, and the $\sigma$-admissibility of $q_\index$, but is otherwise independent of the polynomial 
  degrees  and the mesh $\mesh_\index$.
\end{corollary}

\begin{proof}
  The starting point are the following two inverse estimates 
  \begin{align}
    \label{eq:foo-1}
    \norm{\meshsize^{1/2}(q_\index+1)^{-1}%\max\{1,q_\index\}^{-1}
    \,\Psi_\index}{L^2(\Gamma)}
    &\lesssim \norm{\Psi_\index}{H^{-1/2}(\Gamma)}
    \quad\text{for all }\Psi_\index\in\PP^\q(\mesh_\index),\\
    \label{eq:foo-2}
    \norm{\meshsize^{1/2}(q_\index+1)^{-1}\nabla_\Gamma V_\index}{L^2(\Gamma)}
    &\lesssim \norm{V_\index}{\H^{1/2}(\Gamma)}
    \quad\text{for all }V_\index\in\widetilde\SS^{\q+1}(\mesh_\index),
  \end{align}
  where the hidden constants depend solely on $\partial\Omega$, $\Gamma$, the $\kappa$-shape regularity
  of $\mesh_\index$, and the $\sigma$-admissibility of $q_\index$. 
  The bound (\ref{eq:foo-1}) is essentially taken from \cite[Thm.~{3.9}]{georgoulis}. 
  However, since the non-trivial 
  interpolation argument is not worked out in \cite[Thm.~{3.9}]{georgoulis} and since \cite[Thm.~{3.9}]{georgoulis} 
  is not concerned with open surfaces $\Gamma$, we present the details in 
  Lemma~\ref{lemma:hpinvest-beweis}. We remark that its proof employs the characterization of fractional
  Sobolev norms in terms of the Aronstein-Slobodeckii norm.
  The bound 
  (\ref{eq:foo-2}) follows also from polynomial inverse estimates and an interpolation argument for 
  spaces of piecewise polynomials, which is non-trivial---see \cite{kmr14} for  details. 
  We also refer to \cite[Proposition~5]{hypsing} for the $h$-version of \eqref{eq:foo-2}, in which the 
  dependence on the polynomial degree $q_h$ is left unspecified. 

  We define a weight function by 
  %$\wght := \meshsize^{1/2}(q_\index+1)^{-1}\max\{1,q_\index\}^{-1}$ resp.\,
  $\wght := \meshsize^{1/2}(q_\index+1)^{-1}$. Note that $\norm{\wght/\meshsize^{1/2}}{L^\infty(\Gamma)}\leq 1$
  and that $\wght$ is $\tau$-admissible, where $\tau$ depends only on $\kappa$ and $\sigma$.
  The combination of \eqref{eq:foo-1} with~\eqref{eq:invest:V}
  leads to \eqref{eq:cor:invest:V}.
  The bound \eqref{eq:foo-2} in conjunction 
  with~\eqref{eq:invest:K} yields 
  \eqref{eq:cor:invest:K}.
\end{proof}

%------------
\subsection{Application to efficiency of residual error estimation}\label{section:eff}
%------------
%------------
\subsubsection{Weakly singular integral equations}
\label{sec:efficiency-weakly-singular}
%------------
The next corollary proves that the estimate~\eqref{eq:invest:V} provides 
stability of $\slp$ and $\dlp'$ in weighted norms 
for subspaces $(1-P_\index)L^2(\Gamma)\subseteq L^2(\Gamma)$, where $P_\index$ is 
some projection operator.  Note that the following corollary is in particular applicable to the 
Galerkin projection onto $\PP^\q(\mesh_\index)$.

\begin{corollary}\label{cor:stabV}
Let $\mesh_\index$ be a regular, $\kappa$-shape regular triangulation of $\Gamma$.
Let $X_\index$ be a closed subspace of $\H^{-1/2}(\Gamma)$ with 
$\PP^0(\mesh_\index)\subseteq X_\index \subset L^2(\Gamma)$. Let $\Pi_\index:L^2(\Gamma)\to X_\index$
be the $L^2$-orthogonal projection onto $X_\index$ and $\P_\index:\H^{-1/2}(\Gamma)\to X_\index\subseteq \H^{-1/2}(\Gamma)$ 
denote an arbitrary $\H^{-1/2}(\Gamma)$-stable projection onto $X_\index$. Then, 
there is a constant $\c{invtilde}>0$ depending only on the $\kappa$-shape regularity 
of $\mesh_\index$, the stability constant of $\P_\index$, 
on  $\partial\Omega$ as well as $\Gamma$ such that for 
all $\phi\in L^2(\Gamma)$ and $P_\index\in\{\Pi_\index,\P_\index\}$ 
\begin{align}\label{eq:stab:V}
 \norm{\meshsize^{1/2}\nabla_\Gamma\slp(1-P_\index)\phi}{L^2(\Gamma)}
+ \norm{\meshsize^{1/2}\dlp'(1-P_\index)\phi}{L^2(\Gamma)}
 &\le \c{invtilde} \norm{\meshsize^{1/2}(1-P_\index)\phi}{L^2(\Gamma)}.
\end{align}
\end{corollary}

\begin{proof}

For arbitrary $w \in H^{1/2}(\partial\Omega)$ we get 
by transformation to the reference element and standard approximation results that
$\norm{(1-\Pi_\index) w}{L^2(T)}^2 \lesssim h(T) \norm{w}{H^{1/2}(T)}^2$, where 
we employ the Aronstein-Slobodeckii norm in the definition of $\|\cdot\|_{H^{1/2}(T)}$.
Hence, by summation over all $T \in \TT_\index$, using the Aronstein-Slobodeckii characterization of $\|\cdot\|_{H^{1/2}(\partial\Omega)}$, 
and then the characterization \eqref{eq:H1/2+s} of the norm $\|\cdot\|_{H^{1/2}(\Gamma)}$, we arrive at 
\begin{align*}
  \norm{h^{-1/2}(1-\Pi_\index) w}{L^2(\Gamma)} \lesssim \norm{w}{H^{1/2}(\Gamma)} \quad\text{for all } w \in H^{1/2}(\Gamma).
\end{align*}
Orthogonality of $\Pi_\index$ and a duality argument then shows (see~\cite[Theorem~4.1]{ccdpr:symm} for the analogous proof on polygonal boundaries.)
\begin{align}
\label{eq:H-1/2-estimate-L^2-projection}
 \norm{(1-\Pi_\index)\phi}{\H^{-1/2}(\Gamma)}
 \lesssim \norm{\meshsize^{1/2}(1-\Pi_\index)\phi}{L^2(\Gamma)}
 \quad\text{for all }\phi\in L^2(\Gamma).
\end{align}
Combining this estimate with the inverse estimate~\eqref{eq:invest:V}
 for $\psi = (1-\Pi_\index)\phi$ and $w_h = h^{1/2}$, we get 
\begin{align*}
 \norm{\meshsize^{1/2}\nabla_\Gamma\slp(1-\Pi_\index)\phi}{L^2(\Gamma)}
 + \norm{\meshsize^{1/2}\dlp'(1-\Pi_\index)\phi}{L^2(\Gamma)}
 \lesssim \norm{\meshsize^{1/2}(1-\Pi_\index)\phi}{L^2(\Gamma)}
 \text{ for all }\phi\in L^2(\Gamma).
\end{align*}
For an $\H^{-1/2}(\Gamma)$-stable projection $\P_\index$, we note that the projection
property of $\P_\index$ implies $(1-\P_\index)(1-\Pi_\index)= (1-\P_\index)$. 
This and elementwise stability of $\Pi_h$ imply, for all $\phi\in L^2(\Gamma)$,
\begin{align*}
 \norm{(1-\P_\index)\phi}{\H^{-1/2}(\Gamma)}
 \lesssim \norm{(1-\Pi_\index)\phi}{\H^{-1/2}(\Gamma)}
 \stackrel{\eqref{eq:H-1/2-estimate-L^2-projection}}{\lesssim} \norm{h^{1/2}(1-\Pi_\index)\phi}{L^{2}(\Gamma)}
 \lesssim \norm{\meshsize^{1/2}\phi}{L^2(\Gamma)}.
\end{align*}
Finally, we use the projection property $(1-\P_\index)^2 = (1-\P_\index)$
and argue as for $\Pi_\index$ to obtain
\begin{align*}
 \norm{\meshsize^{1/2}\nabla_\Gamma\slp(1-\P_\index)\phi}{L^2(\Gamma)}
 + \norm{\meshsize^{1/2}\dlp'(1-\P_\index)\phi}{L^2(\Gamma)}
 \lesssim \norm{\meshsize^{1/2}(1-\P_\index)\phi}{L^2(\Gamma)}
 \text{ for all }\phi\in L^2(\Gamma).
\end{align*}
This concludes the proof.
\end{proof}

One immediate consequence of Corollary~\ref{cor:stabV} is the efficiency of the
weighted residual error estimator $\eta_h$ from~\cite{cc1997,cms}: 
Suppose that $\slp$ is $\H^{-1/2}(\Gamma)$-elliptic (in the case $d = 2$, this can be enforced, for example, by 
the scaling requirement $\diam(\Omega)<1$). For $f\in H^1(\Gamma)$,
let $\phi\in \H^{-1/2}(\Gamma)$ be the unique solution of the weakly singular integral equation $\slp\phi=f$. 
Let $X_h\subset L^2(\Gamma)$ 
be a discrete space which contains at least the piecewise constants, i.e., $\PP^0(\TT_\index)\subseteq X_h$, and let 
$\Phi_h\in X_h$ be the unique Galerkin approximation of $\phi$ in $X_h$, i.e.,
\begin{align}
\label{eq:galerkin-slp}
 \dual{\slp(\phi-\Phi_h)}{\Psi_h}_\Gamma = 0
 \quad\text{for all }\Psi_h\in X_h.
\end{align}
Under these assumptions (and, strictly speaking, for polyhedral $\Gamma$),~\cite{cms} proves the reliability estimate
\begin{align}\label{eq:eta:symm}
 C_{\rm rel}^{-1}\,\norm{\phi-\Phi_h}{\H^{-1/2}(\Gamma)} 
 \le \,\eta_{h,\slp} := \norm{h^{1/2}\nabla_\Gamma(f-\slp\Phi_h)}{L^2(\Gamma)}.
\end{align}
The constant $C_{\rm rel}>0$ depends only on $\Gamma$ and the $\kappa$-shape regularity of $\TT_h$.
The following corollary provides the converse efficiency estimate with respect to some slightly stronger
weighted $L^2$-norm. We note that the additional assumption $\phi = \slp^{-1}f\in L^2(\Gamma)$ is in 
particular satisfied for $\Gamma=\partial\Omega$.

\begin{corollary}[Efficiency of $\eta_{h,\slp}$ for weakly singular integral equations]\label{cor:efficient:symm}
Let $\mesh_\index$ be a regular, $\kappa$-shape regular triangulation of $\Gamma$.
Assume $\phi = \slp^{-1}f\in L^2(\Gamma)$ and let $X_h\subseteq \H^{-1/2}(\Gamma)$ be a closed subspace with $\PP^0(\mesh_\index) \subseteq X_h \subset L^2(\Gamma)$. Let 
$\Phi_h \in X_h$ be given by \eqref{eq:galerkin-slp}. Then 
the weighted residual error estimator from~\eqref{eq:eta:symm} satisfies
\begin{align}\label{dp:efficient:symm}
 \eta_{h,\slp} \le C_{\rm eff}\,\norm{h^{1/2}(\phi-\Phi_h)}{L^2(\Gamma)},
\end{align}
where $C_{\rm eff}=\widetilde C_{\rm inv}>0$ is the constant from Corollary~\ref{cor:stabV}.
\end{corollary}%

\begin{proof}
With the Galerkin projection $\P_h:\H^{-1/2}(\Gamma)\to X_h$ and $\Phi_h = \P_h\phi$, Corollary~\ref{cor:stabV} yields
$\eta_{h,\slp} = \norm{h^{1/2}\nabla_\Gamma\slp(\phi-\Phi_h)}{L^2(\Gamma)}
\lesssim \norm{h^{1/2}(\phi-\Phi_h)}{L^2(\Gamma)}$.
\end{proof}%

\begin{remark}[Stronger efficiency of 2D BEM]
While the efficiency estimate~\eqref{dp:efficient:symm} involves a slightly stronger norm 
on the right-hand side, particular situations (as, e.g., the 2D direct BEM formulation
of the Dirichlet problem~\cite{eps65}) permit to bound $\norm{h^{1/2}(\phi-\Phi_\index)}{L^2(\Gamma)}$
by $\norm{\phi-\Phi_\index}{\H^{-1/2}(\Gamma)}$ up to higher-order terms. In \cite{eps65}, this is achieved
by decomposing $\phi$ in a singular part associated with the vertices of $\Omega$ and a regular part; 
the higher-order terms depend only on the regular part of $\phi$.
\eremk
\end{remark}%

%------------
\subsubsection{Hypersingular integral equations}
\label{sec:efficiency-hypersingular}
%------------
Results similar to Corollary~\ref{cor:stabV} also hold for the double-layer integral operator
$\dlp$ and the hypersingular integral operator $\hyp$. Here, particularly interesting 
choices for the projection $\P_\index$ are Scott-Zhang type 
projections onto~$\widetilde\SS^{\q+1}(\mesh_\index)$; see~\cite{scottzhang} as well as 
the adaptation to BEM in~\cite[Section~3.2]{hypsing}.

\begin{corollary}\label{cor:stabW}
Let $\mesh_\index$ be a regular, $\kappa$-shape regular triangulation of $\Gamma$.
Let $X_\index\subseteq \H^{1/2}(\Gamma)$ be a closed subspace with 
$\widetilde\SS^1(\mesh_\index)\subseteq X_\index \subseteq \H^1(\Gamma)$. Let $\P_\index:\H^{1/2}(\Gamma)\to X_\index$
be an $\H^{1/2}(\Gamma)$-stable projection onto $X_\index$. Then, for all $v\in \H^1(\Gamma)$, 
\begin{align}\label{eq:stab:K}
 \norm{\meshsize^{1/2}\nabla_\Gamma\dlp(1-\P_\index)v}{L^2(\Gamma)}
+ \norm{\meshsize^{1/2}\hyp(1-\P_\index)\phi}{L^2(\Gamma)}
 &\le \c{invtilde} \norm{\meshsize^{1/2}\nabla_\Gamma(1-\P_\index)v}{L^2(\Gamma)}.
\end{align}
The constant $\c{invtilde}>0$ depends only on the $\kappa$-shape regularity 
of $\mesh_\index$, the stability constant of $\P_\index$, 
and $\Gamma$.
\end{corollary}

\begin{proof}
Arguing along the lines of the proof of Corollary~\ref{cor:stabV}, we first consider
the Scott-Zhang projection $J_\index:\H^{1/2}(\Gamma)\to\widetilde\SS^1(\mesh_\index)$
onto $\widetilde\SS^1(\mesh_\index)$. According to~\cite[Lemma~7]{hypsing} (strictly speaking, this 
result is formulated for polygonal boundaries only, but the proof transfers with minor changes to the present case), 
it holds
\begin{align}
\label{eq:cor:stabW-10}
 \norm{(1-J_\index)w}{\H^{1/2}(\Gamma)}
 \lesssim \min_{W_\index\in\widetilde\SS^1(\mesh_\index)}\norm{\meshsize^{1/2}\nabla(w-W_\index)}{L^2(\Gamma)}
 \le \norm{\meshsize^{1/2}\nabla(1-J_\index)w}{L^2(\Gamma)}
\end{align}
for all $w\in\H^1(\Gamma)$. The hidden constant depends only on $\Gamma$ 
and the $\kappa$-shape regularity of $\mesh_\index$.
Combining this with the inverse estimate~\eqref{eq:invest:K}
for $v = (1-J_\index)w$, we arrive at 
\begin{align*}
 \norm{\meshsize^{1/2}\nabla_\Gamma\dlp(1\!-\!J_\index)w}{L^2(\Gamma)}
 + \norm{\meshsize^{1/2}\hyp(1\!-\!J_\index)w}{L^2(\Gamma)}
 \lesssim \norm{\meshsize^{1/2}\nabla_\Gamma(1\!-\!J_\index)w}{L^2(\Gamma)}
 \text{ for all }w\in\H^1(\Gamma).
\end{align*}
The same arguments apply for any $\H^{1/2}$-stable projection
$\P_\index:\H^{1/2}(\Gamma)\to X_\index$, but additionally employ its stability. 
\end{proof}

As in Section~\ref{sec:efficiency-weakly-singular}, an immediate consequence of Corollary~\ref{cor:stabW} 
is the efficiency of the weighted residual error estimator $\eta_h$ from~\cite{cc1997,cmps} for 
the hypersingular integral equation: 
Suppose that
$\H^{1/2}(\Gamma)$ does not contain any characteristic function
$\chi_\omega$ for $\omega\subseteq\Gamma$ (this is in particular satisfied if
$\partial\Omega$ is connected and $\Gamma\subsetneqq\partial\Omega$).
Then, $\hyp:\H^{1/2}(\Gamma)\rightarrow H^{-1/2}(\Gamma)$ is an elliptic
isomorphism.
For $f\in L^2(\Gamma)$, let $u\in \H^{1/2}(\Gamma)$ be the unique solution 
of the hypersingular integral equation $\hyp u=f$. Let $X_h\subseteq \H^1(\Gamma)$ 
be a discrete space which contains at least the piecewise affines, 
i.e., $\widetilde\SS^1(\TT_\index)\subseteq X_h$. In addition, let 
$U_h\in X_h$ be the unique Galerkin approximation of $u$ in $X_h$, i.e.,
\begin{align}
\label{eq:galerkin-hyp}
 \dual{\hyp(u-U_h)}{V_h}_\Gamma = 0
 \quad\text{for all }V_h\in X_h.
\end{align}
Under these assumptions (and, strictly speaking, for polyhedral $\Gamma$),~\cite{cmps} proves the reliability estimate
\begin{align}\label{eq:eta:hypsing}
 C_{\rm rel}^{-1}\,\norm{u-U_h}{\H^{1/2}(\Gamma)} 
 \le \,\eta_{h,\hyp} := \norm{h^{1/2}(f-\hyp U_h)}{L^2(\Gamma)}.
\end{align}
The constant $C_{\rm rel}>0$ depends only on $\Gamma$ and the $\kappa$-shape regularity of $\TT_h$.
The following corollary provides the converse efficiency estimate with respect to some slightly stronger
weighted $H^1$-seminorm.

\begin{corollary}[Efficiency of $\eta_h$ for hypersingular integral equations]\label{cor:efficient:hypsing}
  Let $u = \hyp^{-1}f\in\H^1(\Gamma)$. Let $X_h \subseteq \H^{1/2}(\Gamma)$ 
be closed with $\widetilde \SS^{1}(\mesh_\index) \subseteq X_h \subseteq \H^1(\Gamma)$. Let 
$U_h \in X_h$ be given by \eqref{eq:galerkin-hyp}. Then
the weighted residual error estimator from~\eqref{eq:eta:hypsing} satisfies
\begin{align}
 \eta_{h,\hyp} \le C_{\rm eff}\,\norm{h^{1/2}\nabla_\Gamma(u-U_h)}{L^2(\Gamma)},
\end{align}
where $C_{\rm eff}=\widetilde C_{\rm inv}>0$ is the constant from Corollary~\ref{cor:stabW}.
\end{corollary}%

\begin{proof}
With the Galerkin projection $\P_h:\H^{1/2}(\Gamma)\to X_h$ and $U_h = \P_h u$, Corollary~\ref{cor:stabW} yields
$\eta_{h,\hyp} = \norm{h^{1/2}\hyp(u-U_h)}{L^2(\Gamma)}
\lesssim \norm{h^{1/2}\nabla_\Gamma(u-U_h)}{L^2(\Gamma)}$.
\end{proof}%

\begin{remark}
If $\Gamma=\partial\Omega$ is connected, 
the kernel of $\hyp$ is the space f constant functions on $\Gamma$.
Therefore, $\hyp:H^{1/2}_\star(\partial\Omega) \to
H^{-1/2}_\star(\partial\Omega)$ is an elliptic isomorphism, where
$H^{s}_\star(\partial\Omega) := \set{v\in H^{s}(\partial\Omega)}{\dual{v}{1}_{\partial\Omega}=0}$ for
$|s|\le1$. Recall that $\hyp:H^{s}_\star(\partial\Omega)\to H^{s-1}_\star(\partial\Omega)$ is
an isomorphism for all $0\le s\le 1$. For $f\in H^0_\star(\partial\Omega)$, the solution
$u := \hyp^{-1}f$ thus has additional regularity $u\in H^1_\star(\partial\Omega)$, and
Corollary~\ref{cor:efficient:hypsing} holds accordingly.
\end{remark}%

%------------
\subsubsection{Remarks on the extension to $hp$-BEM}
%------------
The above efficiency statements are formulated for the $h$-version BEM. They do generalize to the $hp$-version. 
\newline 
Since the corresponding reliability estimates have only been formulated for closed surfaces $\Gamma = \partial\Omega$
and affine element maps in \cite{km14}, we restrict the following result to this setting: 
\begin{corollary}
\label{cor:efficiency-hp-BEM}
Let $d \in \{2,3\}$, $\Gamma = \partial\Omega$, and $\mesh_\index$ a regular,
$\kappa$-shape regular triangulation of $\Gamma$. Assume that the element maps are affine. 
Let $q_h$ be a $\sigma$-admissible polynomial degree distribution. 
Then there exists $C > 0$ depending only on $\partial\Omega$, the $\kappa$-shape regularity of $\mesh_\index$, 
and $\sigma$ such that the following holds: 
\begin{enumerate}[(i)]
\item
Let $\slp:H^{-1/2}(\partial\Omega) \rightarrow H^{1/2}(\partial\Omega)$ be an isomorphism.
Let $\phi = \slp^{-1} f$ for some $f \in H^1(\partial\Omega)$. 
Set $X_h:= \PP^\q(\mesh_\index)$ 
and let $\Phi_{hp} \in X_h$ be the Galerkin solution given by \eqref{eq:galerkin-slp}. Then:
\begin{align}
\label{eq:cor:efficiency-hp-BEM-10} 
C^{-1} \|\phi - \Phi_{hp}\|_{H^{-1/2}(\partial\Omega)} & \leq 
\eta_{hp,\slp}:= \|(h/(1+q_h))^{1/2} \nabla_\Gamma (f - \slp \Phi_{hp})\|_{L^2(\partial\Omega)},  \\
\label{eq:cor:efficiency-hp-BEM-20} 
\eta_{hp,\slp} &\leq C \|(h/(1+q_h))^{1/2} (\phi - \Phi_{hp})\|_{L^2(\partial\Omega)}. 
\end{align}
\item
Let $\Gamma$ be connected and $u = \hyp^{-1} f$ for some $f \in H^0_\star(\partial\Omega)$. 
Set $X_h:= \SS^{\q+1}(\mesh_\index)\cap H^1_\star(\partial\Omega)$ 
and let $U_{hp} \in X_h$ be the Galerkin solution given by \eqref{eq:galerkin-hyp}. Then:
\begin{align}
\label{eq:cor:efficiency-hp-BEM-30} 
C^{-1} \|u - U_{hp}\|_{H^{1/2}(\partial\Omega)} & \leq 
\eta_{hp,\hyp}:= \|(h/(1+q_h))^{1/2} (f - \hyp U_{hp})\|_{L^2(\partial\Omega)},  \\
\label{eq:cor:efficiency-hp-BEM-40} 
\eta_{hp,\hyp} &\leq C \Bigl[ \|(h/(1+q_h))^{1/2} \nabla_\Gamma (u - U_{hp})\|_{L^2(\partial\Omega)}  \\
\nonumber 
& \mbox{} \qquad  + \|(h/(1+q_h))^{1/2} (u - U_{hp})\|_{L^2(\partial\Omega)}
\Bigr].
\end{align}
\end{enumerate}
\end{corollary}
\begin{proof}
The reliability bounds \eqref{eq:cor:efficiency-hp-BEM-10}, \eqref{eq:cor:efficiency-hp-BEM-30} 
are taken from \cite[Cor.~{3.9}, Cor.~{3.12}]{km14}. 
\newline 
For the proof of \eqref{eq:cor:efficiency-hp-BEM-20}, we let $\Pi_{hp}$ be the $L^2(\partial\Omega)$-projection
and $\P_{hp}$ be the Galerkin projection.
We note that the analogue of 
\eqref{eq:H-1/2-estimate-L^2-projection} is 
\begin{equation}
\label{eq:H-1/2-estimate-L^2-projection-hp} 
 \|(1 - \P_{hp}) \psi\|_{H^{-1/2}(\partial\Omega)} \lesssim 
\|(1 - \Pi_{hp}) \psi\|_{H^{-1/2}(\partial\Omega)} 
\lesssim \|(h/(1+q_h))^{1/2} \psi\|_{L^2(\partial\Omega)}
\qquad \forall \psi \in L^2(\partial\Omega)
\end{equation}
(cf.~\cite[Thm.~{3.8}]{km14}). Hence, proceeding as in the proof of Corollary~\ref{cor:stabV} 
with $w_h = (h/(1+q_h))^{1/2}$ we get 
\begin{align*}
\|w_h \slp(1 - \P_{hp}) \phi  \|_{L^2(\partial\Omega)} 
&\lesssim  \|(1+q_h)^{-1/2} \|_{L^\infty(\partial\Omega)} 
\|(1 - \P_{hp})^2 \phi\|_{H^{-1/2}(\partial\Omega)} + \|w_h (1 - \P_h) \phi\|_{L^2(\partial\Omega)}  \\
&\stackrel{\eqref{eq:H-1/2-estimate-L^2-projection-hp} }{\lesssim} 
\|(1+q_h)^{-1/2} \|_{L^\infty(\partial\Omega)} 
\|w_h (1 - \P_{hp}) \phi\|_{L^2(\partial\Omega)}
+\|w_h (1 - \P_{hp}) \phi\|_{L^2(\partial\Omega)}.
\end{align*}
The proof of \eqref{eq:cor:efficiency-hp-BEM-40} proceeds along similar lines. The key is the analog
of \eqref{eq:cor:stabW-10}. Combining \cite[Lem.~{3.7}]{km14} and 
the proof of \cite[Lem.~{3.10}]{km14} produces an approximation operator 
$J^\prime_{hp}:H^1(\partial\Omega) \rightarrow \SS^{\q+1}(\mesh_\index)$ with 
\begin{equation*}
\|(1 - J^\prime_{hp}) v\|_{H^{1/2}(\partial\Omega)} \lesssim 
\|(h/(1+q_h))^{1/2} \nabla_\Gamma v\|_{L^2(\partial\Omega)} 
+ \|(h/(1+q_h))^{1/2} v\|_{L^2(\partial\Omega)}. 
\end{equation*}
Finally, an operator $J_{hp}:H^1_\ast(\partial\Omega) \rightarrow X_{hp}$ is then obtained by 
setting $J_{hp} v:= J^\prime_{hp}v - \overline{J^\prime_{hp} v}$, where the overbar denotes
the average over $\partial\Omega$. It is easy to see that $J_{hp}$ has the same approximation properties 
as $J^\prime_{hp}$ on the  space $H^1_\ast(\partial\Omega)$. Proceeding as in the proof of 
Corollary~\ref{cor:stabV} or \ref{cor:stabW} finishes the proof. 
\end{proof}
%%%%%%%%%%%%%%%%%%%%%%%%%%%%%%%%%%%%%%%%%%%%%%%%%%%%%%%%%%%%%%%%%%%%%%%%%%%%%%%

%====================================================================
\section{Far-field and near-field estimates for the simple-layer potential}
%====================================================================
\label{section:invest:aux}

\noindent
The proof of Theorem~\ref{thm:invest} is based on decomposing the pertinent potentials 
into ``far-field'' and ``near-field'' contributions. In the present section, we analyze
the decomposition for the simple-layer potential and provide inverse estimates for both 
components. Section~\ref{sec:nearfield-slp} is concerned with inverse estimates for the 
near-field parts, which essentially follow from scaling arguments, whereas 
Section~\ref{sec:farfield-slp} deals with the far-field part. 
Throughout the section, we let 
\begin{equation}
\label{eq:assumption-on-psi} 
\psi\in L^2(\Gamma) \mbox{ and assume that $\psi$ is 
extended by zero to $\partial\Omega\setminus \overline{\Gamma}$,}
\end{equation}
i.e., we identify $\psi$ with $E_{0,\Gamma} \psi$.

%-------------------------------------------------------------------------
\subsection{Decomposition into near-field and far-field}
\label{sec:notation-slp}
%-------------------------------------------------------------------------
For a parameter $\delta > 0$, we define for each element $T\in\mesh_\index$ the neighborhood
$U_T$ of $T$ by 
\begin{align}\label{dp:UT}
 T \subset U_T := \bigcup_{x\in T} B_{2\delta \meshsize(T)}(x).
\end{align}
Since $\partial\Omega$ is Lipschitz and $\Gamma$ stems from a Lipschitz dissection and by $\kappa$-shape regularity of $\mesh_\index$,
we can fix the parameter $\delta > 0$ and find $M\in\N$ (both $\delta$ and $M$ are independent 
of $\TT_\index$) such that the following two conditions are satisfied:
\begin{enumerate}
\item[(a)]
$\Gamma\cap U_T$ is contained in the patch $\omega_\index(T)$ of $T$ (see~\eqref{def:patch} for the definition), i.e.,
\begin{align}
\label{UT:patch}
  \Gamma\cap U_T \subseteq \omega_\index(T). 
\end{align}
\item[(b)]
The covering $\Gamma\subseteq\bigcup_{T\in\mesh_\index}U_T$ is locally finite with a uniform bound, i.e.,  
\begin{align}\label{UT:overlap}
 \#\set{U_T}{T\in\mesh_\index\text{ and }x\in U_T}\le M
 \quad\text{for all }x\in\R^d.% \text{ and all $\ell\in\N$.}
\end{align}
\end{enumerate}
Finally, we fix a bounded domain $U\subset\R^d$ such that
\begin{align}\label{def:U}
U_T \subset U
\quad\text{for all }T\in\mesh_\index. 
%\mbox{ and all $\ell \in \N$.}
\end{align}
It will be important that $U$ is chosen independently of $\TT_\index$.
To deal with the non-locality of the integral operators, we define
for $T\in\mesh_\index$ the near-field 
 $u_{\slp,T}^\near$ 
and the far-field
$u_{\slp,T}^\far$
of the simple-layer potential $u_\slp := \widetilde\slp\psi$ by
\begin{align}
\label{eq:uslp_far_uslp_near}
 u_{\slp,T}^\near := \widetilde\slp(\psi\chi_{\Gamma\cap U_T})
 \quad\text{and}\quad
 u_{\slp,T}^\far := \widetilde\slp(\psi\chi_{\Gamma\setminus U_T}),
\end{align}
where $\chi_\omega$ denotes the characteristic function of the set $\omega\subseteq\R^d$.
We have the obvious identity
\begin{align}\label{eq:decompnearfar}
 u_\slp = \widetilde\slp\psi =u_{\slp,T}^\near + u_{\slp,T}^\far
 \quad\text{for all }T\in\mesh_\index.
\end{align}
In our analysis, we will treat $u_{\slp,T}^\near$ and $u_{\slp,T}^\far$ separately, starting
with the simpler case of $u_{\slp,T}^\near$. 

%----------------------------------------------------------------------------------------
\subsection{Inverse estimates for the near-field part $u_{\slp,T}^\near$}
\label{sec:nearfield-slp}
%----------------------------------------------------------------------------------------
The near-field parts of a potential can be treated with local arguments and the stability properties of the associated
boundary integral operators.

\begin{lemma}\label{lemma:nearfield:V:1}
There exists a constant $\setc{xxx} > 0$ depending only on $\partial\Omega$, $\Gamma$,
and the $\kappa$-shape regularity of $\mesh_\index$ such that for arbitrary 
$T \in \mesh_\index$ and $\Psi_\index^T \in \PP^0(\mesh_\index)$
with $\supp\left( \Psi_\index^T \right)\subseteq\overline{\omega_\index(T)}$ it holds 
\begin{align*}
\norm{\nabla\widetilde\slp\Psi_\index^T}{L^2(U_T)}
  \leq \c{xxx} %\sum_{T\in\mesh_\index}\norm{\wght/\meshsize^{1/2}}{L^\infty(T)}^2
\norm{\meshsize^{1/2}\Psi_\index^T}{L^2(\omega_\index(T))}.
\end{align*}%
\end{lemma}

\begin{proof}
  We fix an element $T\in\mesh_\index$. We recall that $\Psi_\index^T$ is piecewise constant and compute for $x\in\Omega$
  \begin{align*}
    (\nabla \widetilde\slp\Psi_\index^T)(x)
= \sum_{T'\in\omega_\index(T)}\Psi_\index^T|_{T'}\int_{T'}\nabla_xG(x,y)\,dy
  \quad\text{for all } x \in \R^d\setminus \Gamma.
  \end{align*}
  The number of elements $T'$ in the patch $\omega_\index(T)$ is bounded in terms
  of the shape regularity constant $\kappa$ (cf.\ Lemma~\ref{lemma:patch}). With some constant that depends only on $\kappa$, we bound
  \begin{align}
  \label{eq:foo-30}
    \abs{(\nabla \widetilde\slp\Psi_\index^T)(x)}^2
    \lesssim
    \sum_{T'\in\omega_\index(T)}
    \abs{\Psi_\index^T|_{T'}}^2 \Big(\int_{T'}\big|\nabla_ xG(x,y)\big|\,dy\Big)^2. 
  \end{align}
  Next, we show for elements $T'\subseteq \omega_\index(T)$
  \begin{align}
  \label{eq:scaling-argument-for-nabla-G}
    \int_{U_T}\Big(\int_{T'}\big|\nabla_ xG(x,y)\big|\,dy\Big)^2\,d x \lesssim h(T)^d.
  \end{align}
 This follows from a local Lipschitz parametrization of $\partial\Omega$. We assume that (after possibly a Euclidean
 change of coordinates) that $\{(x^\prime,\Lambda(x^\prime))\colon x^\prime \in B_{2r}(0)\}$ is a part 
 of $\partial\Omega$ that contains $\omega_\index(T)$. 
 The function $\Lambda$ is Lipschitz continuous, and  
 we remark in passing that by \cite[Thm.~{3}, Sect.~{VI}]{stein70} 
 we may assume that $\Lambda:\R^{d-1} \rightarrow \R$ is Lipschitz continuous. 
 (If such a local consideration is not possible, then, since the number of 
 local charts is finite by definition of bounded Lipschitz domains, we must have 
 $\operatorname*{diam} (\omega_\index(T)) = O(1)$
 so that \eqref{eq:scaling-argument-for-nabla-G} is trivially true.) We may also assume that 
 $U_T \subset \{(x^\prime,\Lambda(x^\prime) + t)\,|\, x^\prime \in B_{2r}(0), t \in \R\}$. 
 The key observation is that the mapping $\widetilde \Lambda: \R^{d} \rightarrow \R^{d}$ given by 
 $(x^\prime,t) \mapsto (x^\prime,\Lambda(x^\prime) + t)$ is bilipschitz. 
 We conclude that $\widetilde \Lambda^{-1} U_T =:\widetilde U_T$ and 
 $\widetilde \Lambda^{-1} T^\prime =: \widetilde T^\prime \subseteq B_r(0) \times \{0\}$ satisfy 
for some $x_0$ and some $c > 0$, which depends solely on the bilipschitz mapping $\widetilde \Lambda$, 
\begin{equation*}
\widetilde U_T \subseteq B_{c h(T)} (x_0) \times [-c h(T), c h(T)], 
\qquad \widetilde T^\prime \subseteq B_{c h(T)}(x_0) \times \{0\}. 
\end{equation*}
 Finally, using $\nabla_x G(x,y) \simeq |x - y|^{-(d-1)}$, 
 the definition of the surface integral, and the change of variables
 formula for bilipschitz mappings from~\cite[Sec.~3.3.3]{eg92}, we get
\begin{align*}
  \int_{x \in U_T} \bigg( \int_{y \in T^\prime} &|\nabla_x G(x,y)|\,dy\bigg)^2\,dx 
 \simeq \int_{\widetilde x \in \widetilde U_T} 
\left(\int_{\widetilde y \in \widetilde T^\prime} |\widetilde x - \widetilde y|^{-(d-1)}\,d \widetilde y\right)^2\,d\widetilde x \\
&\lesssim \int_{\xi \in B_{c h(T)}(x_0)} \int_{t=-c h(T)}^{c h(T)} 
\left( \int_{\eta \in B_{c h(T)}(x_0)} \left( |\xi - \eta|^2 + t^2 \right)^{-(d-1)/2} \,d\eta \right)^2\,dt\,d\xi
\\
& \simeq  h(T)^{d} 
\int_{\xi \in B_{c}(x_0)} \int_{t=-c}^{c} 
\left( \int_{\eta \in B_{c}(x_0)} \left( |\xi - \eta|^2 + t^2 \right)^{-(d-1)/2} \,d\eta \right)^2\,dt\,d\xi\\
&\simeq h(T)^{d}, 
\end{align*}
where the last estimate follows by a direct estimation of the integrals, which is independent of $h(T)$. 
We have thus shown \eqref{eq:scaling-argument-for-nabla-G}. 
  Inserting \eqref{eq:scaling-argument-for-nabla-G} in \eqref{eq:foo-30} gives 
  \begin{align*}
    \int_{U_T}\abs{(\nabla \widetilde\slp\Psi_\index^T)(x)}^2\,d x
    &\lesssim \sum_{T'\in\omega_\index(T)}\abs{\Psi_\index^T|_{T'}}^2 \meshsize(T)^d
    \simeq \norm{\meshsize^{1/2}\Psi_\index^T}{L^2(\omega_\index(T))}^2.
\qedhere
  \end{align*}
\end{proof}

\begin{proposition}[Near-field bound for $\widetilde\slp$]
\label{prop:nearfield:V}
Let $w_h$ be a $\sigma$-admissible weight function. 
There exists a 
constant $\setc{nearfield}>0$ depending only on $\partial\Omega$, $\Gamma$, the $\kappa$-shape regularity of $\mesh_\index$,
and $\sigma$, 
such that the near-field part $u_{\slp,T}^\near$ satisfies 
$u_{\slp,T}^\near \in H^1(U)$ and 
$\gamma_0^\interior u_{\slp,T}^\near \in H^1(\Gamma)$ together with 
\begin{align}\label{eq1:nearfield:V}
  \sum_{T\in\mesh_\index}\norm{\wght\nabla_\Gamma \gamma_0^\interior u_{\slp,T}^\near}{L^2(T)}^2
  +
  \sum_{T\in\mesh_\index}\norm{\wght/\meshsize^{1/2}}{L^\infty(T)}^2\norm{\nabla u_{\slp,T}^\near}{L^2(U_T)}^2
  \le \c{nearfield}\,\norm{\wght\psi}{L^2(\Gamma)}^2.
\end{align}
\end{proposition}

\begin{proof}
  The stability~\eqref{def:slp} of $\slp:L^2(\partial\Omega)\rightarrow H^1(\partial\Omega)$ proved in 
  \cite{verchota} gives, for each  $T \in \mesh_\index$, 
  \begin{align*}
   \norm{\nabla_\Gamma \gamma_0^\interior u_{\slp,T}^\near}{L^2(T)}
   \leq \norm{\slp( \psi\chi_{U_T\cap\Gamma})}{H^1(\partial\Omega)}
   \lesssim \norm{\psi\chi_{U_T\cap\Gamma}}{L^2(\partial\Omega)}
   = \norm{\psi}{L^2(U_T\cap\Gamma)}.
  \end{align*}
 Summing the last estimate over all $T \in \mesh_\index$ and using~\eqref{UT:patch}--\eqref{UT:overlap}, and $\sigma$-admissibility of $\wght$, we arrive at 
  \begin{align*}
   \sum_{T\in\mesh_\index}\norm{\wght\nabla_\Gamma \gamma_0^\interior u_{\slp,T}^\near}{L^2(T)}^2
   &\lesssim \sum_{T\in\mesh_\index}\norm{\wght}{L^\infty(T)}^2\norm{\psi}{L^2(U_T\cap\Gamma)}^2
   \simeq \norm{\wght\psi}{L^2(\Gamma)}^2,
  \end{align*}%
  where all estimates depend only on the $\kappa$-shape regularity of $\mesh_\index$ and the admissibility constant $\sigma$. This bounds the first term on the left-hand side of~\eqref{eq1:nearfield:V}.
  To bound the second term, let $\Pi_\index$ denote the $L^2(\Gamma)$-orthogonal projection 
  onto $\PP^0(\mesh_\index)$. 
 We decompose the near-field as 
$u_{\slp,T}^\near
= \widetilde \slp( \Pi_\index(\psi\chi_{\Gamma\cap U_T}))
+ \widetilde \slp\big( (1- \Pi_\index)\psi\chi_{\Gamma\cap U_T}\big)$.
The condition $\supp(\psi\chi_{\Gamma\cap U_T})\subseteq\overline{\omega_\index(T)}$
implies $\supp\left(\Pi_\index(\psi\chi_{\Gamma\cap U_T})\right) \subseteq \overline{\omega_\index(T)}$ and therefore, 
  taking $\Psi_\index^T = \Pi_\index(\psi\chi_{\Gamma\cap U_T})$ in Lemma~\ref{lemma:nearfield:V:1}, we conclude
  \begin{align}
  \label{eq:foo-10}
  \begin{split}
  \sum_{T\in\mesh_\index}\norm{\wght/\meshsize^{1/2}} {L^\infty(T)}^2&\norm{\nabla\widetilde\slp(\Pi_\index(\psi\chi_{\Gamma\cap U_T}))}{L^2(U_T)}^2\\
      &\lesssim \sum_{T\in\mesh_\index}\norm{\wght/\meshsize^{1/2}}{L^\infty(T)}^2\norm{\meshsize^{1/2}\Pi_\index(\psi\chi_{\Gamma\cap U_T})}{L^2(\omega_\index(T))}^2\\
    &\lesssim \norm{\wght\psi}{L^2(\Gamma)}^2,
    \end{split}
  \end{align}
where we used the local $L^2$-stability of $\Pi_\index$ in the last estimate.
Recalling the stability 
$\widetilde\slp:\H^{-1/2}(\partial\Omega)\to H^1(U)$ of \eqref{eq:mapping-properties-potentials}, 
the equality \eqref{eq:normeqiv}, 
and the approximation property \eqref{eq:H-1/2-estimate-L^2-projection} of $\Pi_\index$, we get 
\begin{align}
\label{eq:foo-20}
\begin{split}
\sum_{T\in\mesh_\index}\norm{\wght/\meshsize^{1/2}}{L^\infty(T)}^2&\norm{\nabla \widetilde\slp\big((1-\Pi_\index) \psi\chi_{\Gamma\cap U_T}\big)}{L^2(U_T)}^2
\\&
\stackrel{\eqref{eq:mapping-properties-potentials}}{ \lesssim} 
\sum_{T\in\mesh_\index}\norm{\wght/\meshsize^{1/2}}{L^\infty(T)}^2
                       \norm{(1-\Pi_\index) \psi\chi_{\Gamma\cap U_T}}{\H^{-1/2}(\Gamma)}^2
\\ &
\stackrel{\eqref{eq:H-1/2-estimate-L^2-projection}}{\lesssim} \sum_{T\in\mesh_\index}
   \norm{\wght/\meshsize^{1/2}}{L^\infty(T)}^2\norm{\meshsize^{1/2}(\psi\chi_{\Gamma\cap U_T})}{L^2(\Gamma)}^2
\simeq \norm{\wght\psi}{L^2(\Gamma)}^2.
\end{split}
\end{align}
Combining~\eqref{eq:foo-10}--\eqref{eq:foo-20}, we bound
$\sum_{T \in \mesh_\index} \norm{\wght/\meshsize^{1/2}}{L^\infty(T)}^2\|\nabla u_{\slp,T}^\near\|^2_{L^2(U_T)}$ to conclude
the estimate in~\eqref{eq1:nearfield:V}.
\end{proof}

%-------------------------------------------------------------------------
\subsection{Estimates for the far-field part $u_{\slp,T}^\far$}
\label{sec:farfield-slp}
%-------------------------------------------------------------------------
The following lemma is taken from~{\cite{fkmp}}.
For the convenience of the reader and since the same argument underlies 
the proof of the analogous lemma for
the double-layer potential 
(Lemma~\ref{lemma:caccioppoli:K} below), we recall its proof here.

\begin{lemma}[Caccioppoli inequality for $u_{\slp,T}^\far$]
\label{lemma:caccioppoli:V}
With $\Omega^\e=\R^d\setminus\overline\Omega$, the 
function $u_{\slp,T}^\far$ from
\eqref{eq:uslp_far_uslp_near} satisfies $u_{\slp,T}^\far|_\Omega \in C^\infty(\Omega)$,
$u_{\slp,T}^\far|_{\Omega^\e} \in C^\infty(\Omega^\e)$,
and $u_{\slp,T}^\far|_{U_T} \in C^\infty(U_T)$. Moreover, there exists a constant $\setc{caccioppoli}>0$ depending only on $\partial\Omega$, $\Gamma$, and the 
$\kappa$-shape regularity of $\mesh_\index$ such that Hessian matrix $D^2 u_{\slp,T}^\far$ satisfies 
\begin{align}\label{eq:caccioppoli:V}
  \norm{D^2u_{\slp,T}^\far}{L^2(B_{\delta \meshsize(T)}(x))}
  \le \c{caccioppoli}\,\frac{1}{\meshsize(T)}\,\norm{\nabla u_{\slp,T}^\far}{L^2(B_{2\delta \meshsize(T)}(x))}
  \quad\text{for all }x\in T\in \mesh_\index.
\end{align}
\end{lemma}

\begin{proof}
  The statements $u_{\slp,\el}^\far|_\Omega \in C^\infty(\Omega)$ and 
  $u_{\slp,\el}^\far|_{\Omega^\e} \in C^\infty(\Omega^\e)$ are taken from \cite[Theorem~{3.1.1}]{ss},
  and we therefore focus on %the statements on 
  $u_{\slp,\el}^\far|_{U_T} \in C^\infty(U_T)$ and the estimate 
\eqref{eq:caccioppoli:V}. According to~\cite[Proposition~{3.1.7}]{ss},~\cite[Theorem~{3.1.16}]{ss}, 
and~\cite[Theorem~{3.3.1}]{ss}, the function 
$u_{\slp,T}^\far\in H^1_{\loc}(\R^d):=\set{v:\R^d\to\R}{v|_K\in H^1(K)\text{ for all }K\subset\R^d
\text{ compact}}$
solves the transmission problem
  \begin{align}\label{lem:farfield:eq:transmission}
  \begin{array}{rcll}
    -\Delta u_{\slp,T}^\far &=& 0 \quad &\text{in } \Omega \cup \Omega^\e, \\{}
    [u_{\slp,T}^\far] &=& 0 &\text{in } H^{1/2}(\partial\Omega),\\{}
    [\gamma_1 u_{\slp,T}^\far] &=& -\psi\chi_{\Gamma\setminus U_T}
    &\text{in } H^{-1/2}(\partial\Omega).
  \end{array}
  \end{align}
In particular, \eqref{lem:farfield:eq:transmission} states that the jump of 
$u_{\slp,T}^\far$ as well as the jump of the normal derivative vanish
on $\partial\Omega \cap U_T$. This implies that $u_{\slp,T}$ is harmonic in $U_T$
by the following classical argument: First, we observe that $u_{\slp,T}^\far$ 
is distributionally harmonic in $U_T$, since a two-fold  
integration by parts that uses these jump conditions shows 
for 
$v \in C^\infty_0(U_T)$ that 
$\langle u_{\slp,T}^\far,-\Delta v\rangle_\Omega = 0$.
Weyl's lemma (see, e.g., \cite[Theorem~{2.3.1}]{morrey66}) then implies that 
$u_{\slp,T}^\far$ is therefore strongly harmonic and $
u_{\slp,T}^\far \in C^\infty(U_T)$. 

The Caccioppoli inequality \eqref{eq:caccioppoli:V} now expresses 
  interior regularity for elliptic problems. Indeed, for each $u\in H^1(B_{r+\eps})$ such that $u\in H^2(B_r)$ and $\Delta u=f$ on $B_{r+\eps}$
  with balls
  ${B}_r\subseteq B_{r+\eps}$
  with radii $0<r<r+\eps$ and some $f\in L^2(B_{r+\eps})$,
  \cite[Lemma~{5.7.1}]{morrey66} shows 
  \begin{align}
\label{est:inverse_est_cf_Morrey}
   \norm{D^2 u}{L^2(B_r)}
   \lesssim \Big( \norm{f}{L^2(B_{r+\eps})}
   + \frac{1}{\eps}\,\norm{\nabla u}{L^2(B_{r+\eps})}
   + \frac{1}{\eps^2}\norm{u}{L^2(B_{r+\eps})}\Big).
  \end{align}
The hidden constant depends solely on the spatial dimension and is independent of 
  $r,\eps>0$, and $u$, $f$.
  We apply \eqref{est:inverse_est_cf_Morrey} with 
$r =\delta \meshsize(T) = \eps$,
$f = 0$, and $u = u_{\slp,T}^\far-c_T$,
where $c_T = \frac{1}{|B_{2\delta \meshsize(T)}( x)|}\int_{B_{2\delta \meshsize(T)}( x)}u_{\slp,T}^\far(y)\,dy$.
An additional Poincar\'e inequality finally leads to \eqref{eq:caccioppoli:V}. 
Note that $\delta$ and hence $\c{caccioppoli}$ depend only on 
$\partial\Omega$, $\Gamma$, and the $\kappa$-shape regularity of $\TT_\index$.
\end{proof}

The non-local character of the operator $\widetilde \slp$ is represented by the
far-field part. 
Lemma~\ref{lemma:caccioppoli:V} allows us to show a local inverse estimate 
for the far-field part of the simple-layer operator: 

\begin{lemma}[Local far-field bound for $\widetilde\slp$]\label{lemma:farfield:V}
  For all $T\in\mesh_\index$, it holds
\begin{align}
\label{eq:farfield:V}
\norm{\meshsize^{1/2}\nabla_\Gamma \gamma_0^\interior u_{\slp,T}^\far}{L^2(T)}
\le \norm{\meshsize^{1/2}\nabla u_{\slp,T}^\far}{L^2(T)}
\le \c{farfield}\,\norm{\nabla u_{\slp,T}^\far}{L^2(U_T)}. 
\end{align}
  The constant $\setc{farfield}>0$ depends only on $\Gamma$, $\partial\Omega$ and the $\kappa$-shape regularity constant of $\mesh_\index$.
\end{lemma}

\begin{proof}
By Lemma~\ref{lemma:caccioppoli:V} we have $u_{\slp,T}^\far \in C^\infty(U_T)$. 
The first estimate in \eqref{eq:farfield:V} follows from the fact
that, for smooth functions, the surface gradient $\nabla_\Gamma(\cdot)$ is the orthogonal projection 
of the gradient $\nabla(\cdot)$ onto the tangent plane, i.e.,
$\nabla_\Gamma \gamma_0^\interior u(x) = \nabla u(x) - \big(\nabla u(x)\cdot\normal(x)\big)\,\normal(x)$, see~\cite{verchota}.

The second estimate in~\eqref{eq:farfield:V} is proved with a trace inequality and the 
Caccioppoli inequality~\eqref{eq:caccioppoli:V} in the following way. 
We fix an element $T \in \mesh_\index$.
\newline
{\em 1.~step:} We provide a trace inequality. Let $B = B_r(x)$ be a ball with 
center $x \in T \subseteq \partial\Omega$ and radius $r > 0$. 
Let $B^\prime = B_{3r/2}(x)$ and $B^{\prime\prime}:= B_{5 r/4}(x)$. 
We define a smooth cut-off function $\widetilde \chi_B \in C^\infty_0(\R^d)$ with 
$\operatorname*{supp} \widetilde \chi_B \subseteq B^\prime$ and $\widetilde \chi_B \equiv 1$ on $B$ by  
$$
\widetilde \chi_B:= \chi_{B^{\prime\prime}} \star \rho_{r/4}, 
$$
where $\rho_\varepsilon(x)$ is a standard mollifier of the 
form $\rho_\varepsilon(x) = \varepsilon^{-d} \rho_1(x/\eps)$ for a fixed $\rho_1 \in C^\infty_0(\R^d)$
with $\rho_1 \ge 0$, $\operatorname*{supp} \rho_1 \subseteq B_1(0)$ and $\int_{\R^d} \rho_1(x)\,dx = 1$. 
We note that for a $C > 0$ depending solely on the choice of $\rho_1$, we have 
\begin{equation*}
\|\nabla \widetilde \chi_B \|_{L^\infty(\R^d)} \leq C r^{-1}. 
\end{equation*}
With this cut-off function in hand, we estimate for sufficiently regular functions $v$ and the standard
multiplicative trace inequality for $\partial\Omega$
\begin{align}
\nonumber 
\|v\|^2_{L^2(B \cap \partial\Omega)} & \leq \|\widetilde \chi_B v\|^2_{L^2(\partial\Omega)} \lesssim 
\|\widetilde \chi_B v\|^2_{L^2(\Omega)} +  
\|\widetilde \chi_B v\|_{L^2(\Omega)}   
\|\nabla (\widetilde \chi_B v)\|_{L^2(\Omega)}  \\
\label{eq:foo-100}
&\lesssim r^{-1} \|v\|^2_{L^2(B^\prime)} + 
\|v\|_{L^2(B^\prime)} \|\nabla v\|_{L^2(B^\prime)}. 
\end{align}
{\em 2.~step:}
The set $\mathcal{F} := \left\{ \overline{B_{\delta \meshsize(T)/2}(x)} \mid x \in T \right\}$ 
is a closed cover of $T$ with
$\sup_{B \in \mathcal{F}}\diam(B) < \infty$, and $T$ is the set of their midpoints. According to 
Besicovitch's covering theorem, cf.~\cite[Sect.~{1.5.2}]{eg92},
there is a constant $N_d$, which depends only on the spatial dimension $d$, as well as countable subsets 
$\mathcal{G}_j\subseteq\mathcal{F}$, $j=1, \dots, N_d$, the elements of every $\mathcal{G}_j$ being pairwise disjoint, 
such that $T \subseteq \bigcup_{j=1}^{N_d}\bigcup_{B \in \mathcal{G}_j} B$.
Let $\widehat{\mathcal{G}}_j$ be the set of balls obtained by doubling the radius of the balls 
of $\mathcal{G}_j$, i.e., $\widehat{\mathcal{G}}_j := \left\{ B_{\delta \meshsize(T)}(x) 
\mid B_{\delta \meshsize(T)/2}(x) \in \mathcal{G}_j \right\}$. 
As the elements of $\mathcal{G}_j$ are pairwise disjoint and all balls have the same radius $\delta h(T)/2$, 
there is a constant $\widehat N_d$, also depending only on the spatial dimension $d$, 
such that at most $\widehat N_d$ elements of $\widehat{\mathcal{G}}_j$ overlap.
If we write $B := B_{\delta \meshsize(T)/2}(x)$, $B^\prime:= B_{3/4 \delta  \meshsize(T)}(x)$, 
and $\widehat B := B_{\delta \meshsize(T)}(x)$, the multiplicative trace inequality~\eqref{eq:foo-100} 
and the Caccioppoli inequality~\eqref{eq:caccioppoli:V} show 
\begin{align*}
\norm{\nabla u_{\slp,T}^\far}{L^2(B\cap T)}^2
\overset{\eqref{eq:foo-100}}{\lesssim} \frac{1}{\meshsize(T)} \norm{\nabla u_{\slp,\el}^\far}{L^2(B^\prime)}^2 +
\norm{\nabla u_{\slp,\el}^\far}{L^2(B^\prime)}\norm{D^2u_{\slp,\el}^\far}{L^2(B^\prime)}
\overset{\eqref{eq:caccioppoli:V}}{\lesssim} \frac{1}{\meshsize(T)} \norm{\nabla u_{\slp,\el}^\far}{L^2(\wat B)}^2 .
\end{align*}
{\em 3.~step:} We use the last estimate to get
\begin{align*}
\norm{\nabla u_{\slp,\el}^\far}{L^2(T)}^2
\le \sum_{j=1}^{N_d} \sum_{B \in \mathcal{G}_j}\norm{\nabla u_{\slp,\el}^\far}{L^2(B\cap T)}^2
\lesssim \frac{1}{\meshsize(T)}\,\sum_{j=1}^{N_d}\sum_{\wat B \in \widehat{\mathcal{G}}_j}\norm{\nabla u_{\slp,\el}^\far}{L^2(\wat B)}^2
\lesssim \frac{N_d \widehat{N}_d}{\meshsize(T)}\, \norm{\nabla u_{\slp,\el}^\far}{L^2(U_T)}^2.
\end{align*}
This concludes the proof of~\eqref{eq:farfield:V}.
\end{proof}

Summation of the elementwise estimates of Lemma~\ref{lemma:farfield:V} yields 
the following result:

\begin{proposition}[Far-field bound for $\widetilde\slp$]\label{prop:farfield:V}
There is a constant $\c{farfield}>0$ depending only on $\partial\Omega$, $\Gamma$, the $\kappa$-shape regularity 
of $\mesh_\index$, and the $\sigma$-admissibility of the weight function
$w_\index$ such that 
\begin{align*}
  \sum_{T\in\mesh_\index}\norm{\wght\nabla_\Gamma \gamma_0^\interior u_{\slp,T}^\far}{L^2(T)}^2
  &\le \sum_{T\in\mesh_\index}\norm{\wght\nabla u_{\slp,T}^\far}{L^2(T)}^2\\
  &\leq \c{farfield} \left(\norm{\wght/\meshsize^{1/2}}{L^\infty(\Gamma)}^2\norm{\psi}{\H^{-1/2}(\Gamma)}^2 + \norm{\wght\psi}{L^2(\Gamma)}^2\right).
\end{align*}
\end{proposition}

\begin{proof}
  We use the local far-field bound~\eqref{eq:farfield:V}
  of Lemma~\ref{lemma:farfield:V} and $u_{\slp,T}^\far = \widetilde\slp\psi - u_{\slp,T}^\near$ to see
\begin{align}\label{eq2:invest}
%\begin{split}
&
\sum_{T\in\mesh_\index}\norm{\wght\nabla_\Gamma \gamma_0^\interior u_{\slp,T}^\far}{L^2(T)}^2
\leq \sum_{T\in\mesh_\index}\norm{\wght\nabla u_{\slp,T}^\far}{L^2(T)}^2
\overset{\eqref{eq:farfield:V}}{\lesssim} \sum_{T\in\mesh_\index}\norm{\wght/\meshsize^{1/2}}{L^\infty(T)}^2 \norm{\nabla u_{\slp,T}^\far}{L^2(U_T)}^2
\nonumber\\ & \quad
\overset{\eqref{eq:decompnearfar}}{\lesssim} \sum_{T\in\mesh_\index}\norm{\wght/\meshsize^{1/2}}{L^\infty(T)}^2 \norm{\nabla \widetilde\slp\psi}{L^2(U_T)}^2
+\sum_{T\in\mesh_\index}\norm{\wght/\meshsize^{1/2}}{L^\infty(T)}^2 \norm{\nabla u_{\slp,T}^\near}{L^2(U_T)}^2.
%\end{split}
\end{align}
The first term on the right-hand side in~\eqref{eq2:invest} is estimated by stability of $\widetilde\slp$, the finite overlap property~\eqref{UT:overlap}, and~\eqref{eq:normeqiv}
  \begin{align*}
    \sum_{T\in\mesh_\index}\norm{\wght/\meshsize^{1/2}}{L^\infty(T)}^2\norm{\nabla \widetilde\slp\psi}{L^2(U_T)}^2
    \overset{\eqref{UT:overlap}}{\lesssim} \norm{\wght/\meshsize^{1/2}}{L^\infty(\Gamma)}^2\norm{\nabla \widetilde\slp\psi}{L^2(U)}^2
    \lesssim \norm{\wght/\meshsize^{1/2}}{L^\infty(\Gamma)}^2\norm{\psi}{\H^{-1/2}(\Gamma)}^2.
  \end{align*}
  The second term in~\eqref{eq2:invest} is bounded with the near-field bound~\eqref{eq1:nearfield:V}.
\end{proof}

%%%%%%%%%%%%%%%%%%%%%%%%%%%%%%%%%%%%%%%%%%%%%%%%%%%%%%%%%%%%%%%%%%%%%%%%%%%%%%%

%-------------------------------------------------------------------------
\section{Far-field and near-field estimates for the double-layer potential}
\label{sec:aux-dlp}
%-------------------------------------------------------------------------

%------------------------------------------------------------------------------

\noindent
Section~\ref{section:invest:aux} studied far-field and near-field estimates
for the simple-layer potential. 
Corresponding results for the double-layer potential are derived in the present 
section. Throughout this section, let 
\begin{equation*}
v\in \H^1(\Gamma) \subset H^1(\partial\Omega). 
\end{equation*}
In particular, $v\in\H^{1/2}(\Gamma)$ with
$\norm{v}{\H^{1/2}(\Gamma)} = \norm{v}{H^{1/2}(\partial\Omega)}$. Since 
$H^{1/2}$ does not allow jumps, the splitting into near-field and far-field contribution of the double-layer potential $u_\dlp := \widetilde\dlp v$ cannot be done by characteristic functions, but requires smoother cut-off 
functions and greater technical care.

%-------------------------------------------------------------------------
\subsection{Decomposition into near-field and far-field}
%-------------------------------------------------------------------------
\label{sec:notation-dlp}
We use the notation introduced in Section~\ref{sec:notation-slp} concerning the neighborhoods $U_T$. 
In order to define the near-field and far-field parts for 
the double-layer potential, we need appropriate cut-off functions: 
For each $T \in \TT_\index$, we define $\eta_T \in C^\infty_0(\R^d)$ with 
the aid of the standard mollifier $\rho_\varepsilon$ that was already used in the 
proof of Lemma~\ref{lemma:farfield:V}: 
\begin{equation}
\label{eq:definition-eta_T}
\eta_T:= \chi_{\widetilde U_T} \star \rho_{\delta/4 h(T)}, 
\qquad \widetilde U_T:= \cup_{x \in T} B_{\delta/2 h(T)},
\qquad U^\prime_T:= \cup_{x \in T} B_{\delta/4 h(T)}.  
\end{equation}
This function satisfies: 
\begin{equation}
\label{eta:properties}
\operatorname*{supp} \eta_T \subseteq U_T, 
\quad 
\eta_T|_{U^\prime_T} \equiv 1, 
\quad \|\eta_T\|_{L^\infty(\R^d)} \leq 1, 
\quad \|\nabla \eta_T\|_{L^\infty(\R^d)} \lesssim \frac{1}{h(T)}, 
\end{equation}
where the implied constant depends on the  $\kappa$-shape regularity of the triangulation through
the parameter $\delta$. 
We note that the assumptions on $U_T$ imply
$(\operatorname*{supp} \eta_T)\cap \Gamma  \subseteq \overline{\omega_\index(T)}$.

The following lemma may be viewed as an extension of \cite[Thm.~{7.1}]{ds80} to the  case of curved elements.

\begin{lemma}[Poincar\'e-Friedrichs inequality on patches]
\label{lemma:poincare}
  Let $v\in\H^1(\Gamma)$. For each $T\in\mesh_\index$, there is a constant $v_T\in\mathbb{R}$ 
  such that $(v-v_T)\eta_T\in\H^1(\Gamma)$, $(v-v_T)(1-\eta_T)\in H^1(\partial\Omega)$, 
  and
  \begin{align}
  \label{est:poincare on patch - first est}
  \norm{v-v_T}{L^2(\omega_\index(T))}
  &
  \le \c{poincare} \norm{\meshsize\nabla_\Gamma v}{L^2(\omega_\index(T))}, 
  \\
  \label{est:poincare on patch - second est}
  \norm{(v-v_T)\eta_\el}{H^{1/2}(\partial\Omega)}
  &
  \le \c{poincare} \norm{\meshsize^{1/2}\nabla_\Gamma v}{L^2(\omega_\index(\el))},
  \\
  \label{est:poincare on patch - third est}
  \norm{(v-v_T)\eta_\el}{H^{1}(\partial\Omega)}
  &\le \c{poincare} \norm{\nabla_\Gamma v}{L^2(\omega_\index(\el))}.
  \end{align}
  The constant $v_T$ satisfies $v_T = 0$ if 
    $\partial\omega_\index(T)\cap\partial\Gamma$ contains a facet of the triangulation.
  The constant $\setc{poincare}>0$ depends only on 
  $\partial\Omega$ and the $\kappa$-shape regularity constant of $\mesh_\index$. 
\end{lemma}
\begin{proof}

It is clear that $(v -v_T) (1 - \eta_T) \in H^1(\partial\Omega)$, since $v \in \H^1(\Gamma)$ and $\eta_T$ is smooth. 
The remaining statements require more care. 
\newline 
{\em 1.~step:} For $v \in \H^1(\Gamma)$ and a facet $f\in\FF_\index$ of the triangulation $\TT_\index$ 
(recall that facets are images of $(d-2)$-faces of $\Tref$ under the element map) denote by 
$\ell_{f}(v)$ the average of $v$ on $f$. As $\ell_f(1) = 1$, we can use the Deny-Lions lemma on
the reference element, and the assumptions on the element maps then imply
\begin{align}
\label{eq:poincare-100}
\|v - \ell_f(v)\|_{L^2(T)} &\lesssim h(T) \|\nabla_\Gamma v\|_{L^2(T)} \qquad \mbox{ if $f$ is a facet of $T$,} \\
\label{eq:poincare-200}
|\ell_{f_1}(v) - \ell_{f_2}(v)| &\lesssim h(T)^{1-(d-1)/2} \|\nabla_\Gamma v\|_{L^2(T)} 
\qquad \mbox{ if $f_1$, $f_2$ are two facets of $T$}. 
\end{align}
{\em 2.~step:} Fix an element $T\in\mesh_\index$. 
\begin{itemize}
\item 
If $\eta_T|_{\partial\Gamma}  \equiv 0$, then select an 
arbitrary facet $f_T$ of the element patch $\omega_\index(T)$. 
\item 
If $\eta_T|_{\partial\Gamma} \not\equiv 0$, 
then we claim that there exists a facet $f$ of $\omega_\index(T)$ with $f \subseteq \partial\Gamma$. 
To see this, let $x_0 \in \partial\Gamma$ with $\eta_T(x_0)\ne 0$. By continuity of $\eta_T$ and since
$\partial\Gamma$ is covered by facets of the triangulation, we may assume that $x_0$ is 
in the interior of a boundary facet $f_T$. This facet belongs to a unique element $T_f$ of the triangulation; 
by continuity of $\eta_T$, we may assume $\operatorname*{supp} \eta_T \cap T_f \neq\emptyset$. Since 
$(\operatorname*{supp} \eta_T) \cap \overline{\Gamma} \subseteq \overline{\omega_h(T)}$, we conclude
$T_f \subseteq \omega_\index(T)$ and thus the boundary facet $f_T$ is a facet of $\omega_h(T)$. 
\end{itemize}
Set $v_T:= \ell_{f_T}(v)$. An immediate consequence of $v \in \H^1(\Gamma)$ is that 
$v_T = 0$ if $\eta_T$ does not vanish on $\partial\Gamma$. Since $\eta_T$ is smooth, we conclude 
$(v - v_T)\eta_T \in \H^1(\Gamma)$. In fact, viewed as a function on $\partial\Omega$, we have 
\begin{equation}
\label{eq:poincare-300} 
\operatorname*{supp} ((v - v_T) \eta_T) \subseteq \overline{\omega_\index(T)}. 
\end{equation}
\newline 
{\em 3.~step:}
The bounds \eqref{eq:poincare-100}, \eqref{eq:poincare-200} in conjunction
with Lemma~\ref{lemma:patch} imply 
\begin{align}
\label{eq:poincare-400}
\|v - v_T\|_{L^2(\omega_\index(T))} &\lesssim \|h \nabla_\Gamma v\|_{L^2(\omega_\index(T))}, \\
\label{eq:poincare-500}
\|\nabla_\Gamma(v - v_T)\|_{L^2(\omega_\index(T))} &= \|\nabla_\Gamma v\|_{L^2(\omega_\index(T))},
\end{align}
where~\eqref{eq:poincare-400} is already the claim~\eqref{est:poincare on patch - first est}. 
The product rule, \eqref{eta:properties}, \eqref{eq:poincare-300}, the estimate~\eqref{eq:poincare-400},
the trivial bound $\meshsize(T) \lesssim |T|^{1/(d-1)}\le |\Gamma|^{1/(d-1)} \lesssim 1$
yield
  \begin{align*}
   \norm{\nabla_\Gamma\big((v-v_T)\eta_T\big)}{L^2(\partial\Omega)}
   \le \norm{(v-v_T)\nabla_\Gamma\eta_T}{L^2(\omega_\index(T) )}
   + \norm{\eta_T\nabla_\Gamma(v-v_T)}{L^2(\omega_\index(T))}
   \lesssim\norm{\nabla_\Gamma v}{L^2(\omega_\index(\el))},
  \end{align*}
  which proves~\eqref{est:poincare on patch - third est}.
  It remains to verify~\eqref{est:poincare on patch - second est}. To that end, we recall the interpolation inequality
  $\norm{u}{H^{1/2}(\partial\Omega)}^2 \lesssim \norm{u}{L^2(\partial\Omega)}\norm{u}{H^1(\partial\Omega)}$ for all %
  $u\in H^1(\partial\Omega)$.
  Since $\norm{(v-v_T)\eta_T}{L^2(\partial\Omega)} \le \norm{v-v_T}{L^2(\omega_\index(T))}$, we get
  \begin{align*}
   \norm{(v-v_T)\eta_\el}{H^{1/2}(\partial\Omega)}
   &\lesssim \norm{(v-v_T)\eta_\el}{L^2(\partial\Omega)}^{1/2}
   \norm{(v-v_T)\eta_\el}{H^1(\partial\Omega)}^{1/2}
   \\&
   \lesssim \norm{\meshsize\nabla_\Gamma v}{L^2(\omega_\index(\el))}^{1/2}
   \norm{\nabla_\Gamma v}{L^2(\omega_\index(\el))}^{1/2}
   \\&
   \simeq \norm{\meshsize^{1/2}\nabla_\Gamma v}{L^2(\omega_\index(\el))},
  \end{align*}
  where the last estimate hinges on $\kappa$-shape regularity of $\mesh_\index$ 
  (cf.~Lemma~\ref{lemma:patch}, (\ref{item:lemma:patch-i})). 
\end{proof}

Let $v\in \H^1(\Gamma)$. For each $T\in\mesh_\index$, let $v_T$ be the constant from 
Lemma~\ref{lemma:poincare}. For each $T \in \mesh_\index$ we define the near-field and the far-field
part of the double-layer potential $u_\dlp = \widetilde\dlp v$ by 
\begin{align}\label{eq:v_T}
 u_{\dlp,T}^\near := \widetilde\dlp\big((v-v_T)\eta_T\big)
 \quad\text{and}\quad
 u_{\dlp,T}^\far := \widetilde\dlp\big((v-v_T)(1-\eta_T)\big).
\end{align}
Note that $(v-v_T)\eta_T\in\H^1(\Gamma)\subseteq H^1(\partial\Omega)$ and 
$(v-v_T)(1-\eta_T)\in H^1(\partial\Omega)$ so that 
$u_{\dlp,T}^\near,u_{\dlp,T}^\far\in H^1(U\backslash\partial\Omega)$
are well-defined.
Since $\widetilde \dlp 1 \equiv -1$ in $\Omega$ and 
$\widetilde \dlp 1 \equiv 0$ in $\Omega^\e$, we have, for every $\el \in \mesh_\index$,
the identities 
\begin{align}
\label{eq:addition of near and far-field contributions of Ktilde}
\begin{split}
 u_\dlp + v_T &= u_{\dlp,T}^\near + u_{\dlp,T}^\far
 \quad\text{in }\Omega
\quad\text{and}\quad
 u_\dlp = u_{\dlp,T}^\near + u_{\dlp,T}^\far
 \quad\text{in }\Omega^\e.
\end{split}
\end{align}

%---------------------------------------------------------------------------------------------
\subsection{Inverse estimates for the near-field part $u_{\dlp,T}^\near$}
%----------------------------------------------------------------------------------------------------
The following proposition provides an estimate for the near-field part of the double-layer potential.

\begin{proposition}[Near-field bound for $\widetilde\dlp$]
\label{lemma:nearfield:K}
Let $w_h$ be a $\sigma$-admissible weight function. There exists a 
constant $\setc{nearfield}>0$ depending only on $\partial\Omega$, $\Gamma$, the $\kappa$-shape regularity of $\mesh_\index$,
and $\sigma$ such that the near-field part $u_{\dlp,T}^\near$ satisfies 
$\gamma_0^\interior u_{\dlp,T}^\near\in H^1(\Gamma)$, $u_{\dlp,T}^\near|_\Omega\in H^1(\Omega)$,
and $u_{\dlp,T}^\near|_{U\setminus \overline{\Omega}}\in H^1(U \setminus \overline{\Omega})$  with
\begin{align}\label{eq:nearfieldK}
%\begin{split}
&\sum_{T\in\mesh_\index}\norm{\wght/\meshsize^{1/2}}{L^\infty(T)}^2\Big(
\norm{\meshsize^{1/2}\nabla_\Gamma \gamma_0^\interior u_{\dlp,T}^\near}{L^2(T)}^2
+ \norm{\nabla u_{\dlp,T}^\near}{L^2(U_T\cap\Omega)}^2
+ \norm{\nabla u_{\dlp,T}^\near}{L^2(U_T\cap\Omega^\e)}^2
\Big)
\notag\\ & \quad
\le \c{nearfield}\,\norm{\wght\nabla_\Gamma v}{L^2(\Gamma)}^2.
%\end{split}
\end{align}%
\end{proposition}

\begin{proof}
Recall stability~\eqref{def:dlp} of $\trace\widetilde\dlp = \dlp-\frac{1}{2}: H^1(\partial\Omega)\rightarrow H^1(\partial\Omega)$. 
  Taking into account~\eqref{eta:properties}
  and the Poincar\'e-type estimate~\eqref{est:poincare on patch - third est}, we observe
\begin{align*}
\norm{\nabla_\Gamma \gamma_0^\interior u_{\dlp,\el}^\near}{L^2(\el)}
&
\leq \norm{\nabla_\Gamma \gamma_0^\interior u_{\dlp,\el}^\near}{L^2(\Gamma)}
\overset{\eqref{def:dlp}}{\lesssim} \norm{(v-v_T)\eta_\el}{H^1(\partial\Omega)}
\overset{\eqref{est:poincare on patch - third est}}{\lesssim} \norm{\nabla_\Gamma v}{L^2(\omega_\index(T))}.
\end{align*}
  Summation over all $\el\in\mesh_\index$ shows
\begin{align}\label{eq1:nearfield:K}
  \sum_{T\in\mesh_\index}\norm{\wght/\meshsize^{1/2}}{L^\infty(T)}^2\norm{\meshsize^{1/2}\nabla_\Gamma \gamma_0^\interior u_{\dlp,T}^\near}{L^2(T)}^2
  \lesssim \norm{\wght\nabla_\Gamma v}{L^2(\Gamma)}^2.
\end{align}%
Next, we use the continuity of $\widetilde\dlp: H^{1/2}(\partial\Omega) \to H^1(U\setminus\partial\Omega)$ from~\eqref{eq:mapping-properties-potentials} and get 
\begin{align*}
\norm{\nabla u_{\dlp,\el}^\near}{L^2(U_\el\cap\Omega)}^2 +
\norm{\nabla u_{\dlp,\el}^\near}{L^2(U_\el\cap\Omega^\e)}^2
\overset{\eqref{eq:mapping-properties-potentials}}{\lesssim} \norm{(v-v_T)\eta_\el}{H^{1/2}(\partial\Omega)}^2
\overset{\eqref{est:poincare on patch - second est}}{\lesssim} \norm{\meshsize^{1/2}\nabla_\Gamma v}{L^2(\omega_\index(\el))}^2.
  \end{align*}
  Summation over all $\el\in\mesh_\index$ gives
\begin{align}\label{eq2:nearfield:K}
  \sum_{T\in\mesh_\index}\norm{\wght/\meshsize^{1/2}}{L^\infty(T)}^2\Big(\norm{\nabla u_{\dlp,T}^\near}{L^2(U_T\cap\Omega)}^2
  + \norm{\nabla u_{\dlp,T}^\near&}{L^2(U_T\cap\Omega^\e)}^2
  \Big)
  \lesssim \norm{\wght\nabla_\Gamma v}{L^2(\Gamma)}^2.
  \end{align}
  Combining~\eqref{eq1:nearfield:K}--\eqref{eq2:nearfield:K}, we conclude the proof.
\end{proof}
%
%----------------------------------------------------------------------
\subsection{Estimates for the far-field part $u_{\dlp,\el}^\far$}
%----------------------------------------------------------------------
As for the simple-layer potential, we have a Caccioppoli inequality for the double-layer potential, 
which underlies the analysis of the far-field contribution. 
%--------------------------------------------------------------------------------------------------

For the next result, recall $U^\prime_T$ from \eqref{eq:definition-eta_T}. 
\begin{lemma}[Caccioppoli inequality for $u_{\dlp,T}^\far$]
\label{lemma:caccioppoli:K}
For the constant $\c{caccioppoli}$ of Lemma~\ref{lemma:caccioppoli:V},
the functions $u_{\dlp,T}^\far$ of \eqref{eq:v_T} satisfy
  $u_{\dlp,T}^\far|_\Omega \in C^\infty(\Omega)$,
  $u_{\dlp,T}^\far|_{\Omega^\e} \in C^\infty(\Omega^\e)$, and 
  $u_{\dlp,T}^\far|_{U_T^\prime}\in C^\infty(U^\prime_T)$ together with 
  \begin{align}\label{eq:caccioppoli:K}
    \norm{D^2u_{\dlp,T}^\far}{L^2(B_{\delta\meshsize(T)/8}( x))}
     \le \c{caccioppoli}\,\frac{1}{\meshsize(T)}\,\norm{\nabla u_{\dlp,T}^\far}{L^2(B_{\delta \meshsize(T)/4}( x))}
     \quad\text{for all } x\in T\in\mesh_\index.
  \end{align}
\end{lemma}

\begin{proof}
  The proof is very similar to that of Lemma~\ref{lemma:caccioppoli:V}. One observes 
  that the far-field $u_{\dlp,T}^\far$ solves the transmission problem
  \begin{align*}
  \begin{array}{rcll}
    -\Delta u_{\dlp,T}^\far &=& 0 \quad &\text{in } \Omega\cup\Omega^\e,\\{}
    [u_{\dlp,T}^\far] &=& (v-v_T)(1-\eta_T) &\text{in } H^{1/2}(\partial\Omega),\\{}
    [\gamma_1 u_{\dlp,T}^\far] &=& 0 &\text{in }H^{-1/2}(\partial\Omega).
  \end{array}
  \end{align*}
 We note that $(1-\eta_T)|_{\Gamma\cap U_T^\prime}=0$ by construction of $\eta_T$ 
  in (\ref{eta:properties}). 
Hence, the same reasoning as in the proof of 
  Lemma~\ref{lemma:caccioppoli:V} can be applied to reach the conclusion 
  \eqref{eq:caccioppoli:K}. 
\end{proof}

\begin{lemma}[Local far-field bound for $\widetilde\dlp$]\label{lemma:farfield:K}
  For all $T\in\mesh_\index$
\begin{align}
\label{eq:farfield:K}
\norm{\meshsize^{1/2}\nabla_\Gamma \gamma_0^\interior u_{\dlp,T}^\far}{L^2(T)}
\le \norm{\meshsize^{1/2}\nabla u_{\dlp,T}^\far}{L^2(T)}
\le \c{farfield}\,\norm{\nabla u_{\dlp,T}^\far}{L^2(U^\prime_T)}. 
  \end{align}
  The constant $\setc{farfield}>0$ depends only on $\partial\Omega$ and the $\kappa$-shape regularity constant of $\mesh_\index$.
\end{lemma}

\begin{proof}
The lemma is shown in exactly the same way as the corresponding bound for 
the simple-layer potential $\slp$ in Lemma~\ref{lemma:farfield:V}, appealing
to the Caccioppoli inequality \eqref{eq:caccioppoli:K} instead of \eqref{eq:caccioppoli:V}. 
\end{proof}

\begin{proposition}[Far-field bound for $\widetilde\dlp$]\label{prop:farfield:K}
Let $w_h$ be a $\sigma$-admissible weight function. 
  There is a constant $\c{farfield}>0$ depending only on $\partial\Omega$, the $\kappa$-shape regularity constant
of $\mesh_\index$, and $\sigma$ such that 
\begin{align*}
\sum_{\el\in\mesh_\index}\norm{\wght/\meshsize^{1/2}}{L^\infty(T)}^2&\norm{\meshsize^{1/2}\nabla_\Gamma \gamma_0^\interior u_{\dlp,\el}^\far}{L^2(\el)}^2\\
&\leq \sum_{\el\in\mesh_\index}\norm{\wght\nabla u_{\dlp,\el}^\far}{L^2(\el)}^2\\
&\leq \c{farfield} \left(\norm{\wght\nabla_\Gamma v}{L^2(\Gamma)}^2
+ \norm{\wght/\meshsize^{1/2}}{L^\infty(\Gamma)}^2\norm{v}{\H^{1/2}(\Gamma)}^2\right).
  \end{align*}
\end{proposition}
\begin{proof}
Lemma~\ref{lemma:farfield:K} implies 
\begin{align}
\label{prop:invest:K:eq:2}
&\sum_{\el\in\mesh_\index}\norm{\wght\nabla_\Gamma \gamma_0^\interior u_{\dlp,\el}^\far}{L^2(\el)}^2
\leq \sum_{\el\in\mesh_\index}\norm{\wght\nabla u_{\dlp,\el}^\far}{L^2(\el)}^2
\lesssim \sum_{\el\in\mesh_\index}
\norm{\wght/\meshsize^{1/2}}{L^\infty(T)}^2\norm{\nabla u_{\dlp,\el}^\far}{L^2(U_\el^\prime)}^2
\nonumber\\ &
= \sum_{\el\in\mesh_\index}\norm{\wght/\meshsize^{1/2}}{L^\infty(T)}^2\norm{\nabla u_{\dlp,\el}^\far}{L^2(U_\el^\prime\cap\Omega)}^2
+ \sum_{\el\in\mesh_\index}\norm{\wght/\meshsize^{1/2}}{L^\infty(T)}^2\norm{\nabla u_{\dlp,\el}^\far}{L^2(U_\el^\prime\cap\Omega^\e)}^2.
\end{align}
With the identities~\eqref{eq:addition of near and far-field contributions of Ktilde}
 and a triangle inequality, we therefore obtain
\begin{align}
\begin{split}
\sum_{\el\in\mesh_\index}&\norm{\wght\nabla u_{\dlp,\el}^\far}{L^2(\el)}^2\\
&\overset{\eqref{prop:invest:K:eq:2}}{\lesssim} \sum_{\el\in\mesh_\index}
\norm{\wght/\meshsize^{1/2}}{L^\infty(T)}^2\Bigl( \norm{\nabla \widetilde\dlp(v-v_T)}{L^2(U_\el^\prime\cap\Omega)}^2
+ \norm{\nabla \widetilde\dlp v}{L^2(U_\el^\prime\cap\Omega^\e)}^2\Bigr)
\\ &\qquad
+ \sum_{\el\in\mesh_\index}\norm{\wght/\meshsize^{1/2}}{L^\infty(T)}^2\Bigl( \norm{\nabla u_{\dlp,\el}^\near}{L^2(U_\el^\prime\cap\Omega)}^2
+ \norm{\nabla u_{\dlp,\el}^\near}{L^2(U_\el^\prime\cap\Omega^\e)}^2 \Bigr).
\end{split}
\end{align}
The near-field contribution is bounded by Proposition~\ref{lemma:nearfield:K}.
Furthermore, noting $\nabla\widetilde\dlp v_T = \nabla(-v_T)=0$ in $\Omega$, we get 
\begin{align*}
\sum_{\el\in\mesh_\index}&\norm{\wght\nabla u_{\dlp,\el}^\far}{L^2(\el)}^2\\
&
\overset{\eqref{eq:nearfieldK}}{\lesssim} \norm{\wght/\meshsize^{1/2}}{L^\infty(\Gamma)}^2\sum_{\el\in\mesh_\index}
\Bigl( \norm{\nabla \widetilde\dlp v}{L^2(U_\el^\prime\cap\Omega)}^2
+ \norm{\nabla \widetilde\dlp v}{L^2(U_\el^\prime\cap\Omega^\e)}^2 \Bigr)
+ \norm{\wght\nabla_\Gamma v}{L^2(\Gamma)}^2
\\ &
\overset{\eqref{eq:mapping-properties-potentials}}{\lesssim} \norm{w_\index/\meshsize^{1/2}}{L^\infty(\Gamma)}^2\norm{v}{\H^{1/2}(\Gamma)}^2
+ \norm{w_\index\nabla_\Gamma v}{L^2(\Gamma)}^2,
\end{align*}
where we have used continuity \eqref{eq:mapping-properties-potentials} of $\widetilde\dlp$, 
the overlap property \eqref{UT:overlap}, 
and $\norm{v}{\H^{1/2}(\Gamma)}=\norm{v}{H^{1/2}(\partial\Omega)}$.
\end{proof}

%%%%%%%%%%%%%%%%%%%%%%%%%%%%%%%%%%%%%%%%%%%%%%%%%%%%%%%%%%%%%%%%%%%%%%%%%%%%%%%

%-------------------------------------------------------------------------
\section{Proof of Theorem~\ref{thm:invest}}
%-------------------------------------------------------------------------
\label{section:proof}

\noindent
Finally, we are in a position to prove the inverse estimates~\eqref{eq:invest:V}, \eqref{eq:invest:K} 
of Theorem~\ref{thm:invest}.

\begin{proof}[Proof of the inverse estimate~\eqref{eq:invest:V}] 
Let $\psi \in L^2(\Gamma)$, 
extend $\psi$ by zero to the entire boundary
$\partial\Omega$, and recall the notation from Section~\ref{sec:notation-slp}. 
First, we treat the simple-layer potential $\slp$.
With the bounds of 
Propositions~\ref{prop:nearfield:V} and \ref{prop:farfield:V} we get 
\begin{align}\label{eq1:invest}
\begin{split}
\norm{\wght\nabla_\Gamma \slp\psi}{L^2(\Gamma)}^2
&
= \sum_{T\in\mesh_\index} \norm{\wght\nabla_\Gamma \slp\psi}{L^2(T)}^2
\\ &
\lesssim \sum_{T\in\mesh_\index}\norm{\wght\nabla_\Gamma \gamma_0^\interior u_{\slp,\el}^\far}{L^2(T)}^2
+ \sum_{T\in\mesh_\index}\norm{\wght\nabla_\Gamma \gamma_0^\interior u_{\slp,\el}^\near}{L^2(T)}^2 \\
&
\lesssim \norm{\wght/\meshsize^{1/2}}{L^\infty(\Gamma)}^2\norm{\psi}{\H^{-1/2}(\Gamma)}^2
+ \norm{\wght\psi}{L^2(\Gamma)}^2.
\end{split}
\end{align}
The estimate for the adjoint double-layer potential $\dlp'$ follows by similar arguments.
We split the left-hand side into near-field and far-field contributions to obtain
\begin{align}
  \label{eq:invest:Kadj:1}
%  \begin{split}
\hspace*{-2mm}
    \norm{\wght\dlp'\psi}{L^2(\Gamma)}^2
    \lesssim \sum_{\el\in\mesh_\index}&\norm{\wght}{L^\infty(T)}^2\norm{\dlp'(\psi\chi_{U_\el\cap\Gamma})}{L^2(\el)}^2
    +\sum_{\el\in\mesh_\index}\norm{\wght}{L^\infty(T)}^2\norm{\dlp'(\psi\chi_{\Gamma\setminus U_\el})}{L^2(\el)}^2.
%  \end{split}
\end{align}
The continuity $\dlp':L^2(\partial\Omega)\rightarrow L^2(\partial\Omega)$ stated in \eqref{def:adlp}  yields 
for the near-field contribution
\begin{align*}
\sum_{\el\in\mesh_\index}\norm{\wght}{L^\infty(T)}^2\norm{\dlp'(\psi\chi_{U_\el\cap\Gamma})}{L^2(\el)}^2
&\leq \sum_{\el\in\mesh_\index}\norm{\wght}{L^\infty(T)}^2\norm{\dlp'(\psi\chi_{U_\el\cap\Gamma})}{L^2(\partial\Omega)}^2\\
&
\overset{\eqref{def:adlp}}{\lesssim} \sum_{\el\in\mesh_\index}\norm{\wght}{L^\infty(T)}^2\norm{\psi}{L^2(U_\el\cap\Gamma)}^2
\lesssim \norm{\wght\psi}{L^2(\Gamma)}^2.
\end{align*}
For the far-field contribution, we write 
$u_{\slp,T}^\far  = \widetilde \slp (\psi \chi_{\Gamma \setminus U_T})$
and note that $\dlp' = -1/2 + \gint \widetilde\slp$
and clearly $(\psi\chi_{\Gamma\backslash U_T})|_T=0$. Therefore, 
on $T$ we have $\dlp' (\psi \chi_{\Gamma\setminus U_T}) = \gint u_{\slp,T}^\far$. 
Furthermore, by the smoothness of $u_{\slp,T}^\far$ near $T$ (see 
Lemma~\ref{lemma:caccioppoli:V}), we have 
$\gint u_{\slp,T}^\far = \partial_{\normal} u_{\slp,T}^\far$ on $T$
(cf. Remark~\ref{rem:gamma_1})
and get 
\begin{align*}
\norm{\dlp'(\psi\chi_{\Gamma\setminus U_\el})}{L^2(\el)}
= \norm{\gint u_{\slp,\el}^\far}{L^2(\el)}
= \norm{\partial_\normal u_{\slp,\el}^\far}{L^2(\el)}
\lesssim \norm{\nabla u_{\slp,\el}^\far}{L^2(\el)}.
\end{align*}
The far-field contribution in~\eqref{eq:invest:Kadj:1}
can therefore be bounded by Proposition~\ref{prop:farfield:V}
\begin{align*}
\sum_{\el\in\mesh_\index}\norm{\wght}{L^\infty(T)}^2\norm{\dlp'(\psi\chi_{\Gamma\setminus U_\el})}{L^2(\el)}^2
&\lesssim \sum_{\el\in\mesh_\index}\norm{\wght\nabla u_{\slp,\el}^\far}{L^2(\el)}^2\\
&\lesssim \norm{\wght\psi}{L^2(\Gamma)}^2 + \norm{\wght/\meshsize^{1/2}}{L^\infty(\Gamma)}^2\norm{\psi}{\H^{-1/2}(\Gamma)}^2.
\end{align*}
Altogether, this gives
\begin{align*}
\norm{\wght\dlp'\psi}{L^2(\Gamma)}
\lesssim \norm{\wght\psi}{L^2(\Gamma)} + \norm{\wght/\meshsize^{1/2}}{L^\infty(\Gamma)}\norm{\psi}{\H^{-1/2}(\Gamma)}.
\qedhere
\end{align*}
\end{proof}

\begin{proof}[Proof of inverse estimate~\eqref{eq:invest:K}] First, we treat the double-layer potential $\dlp$.
Let $v \in \H^1(\Gamma)$, extend $v$ by zero to $v\in H^1(\partial\Omega)$,
and recall the notation from Section~\ref{sec:notation-dlp}.
We recall the stability of $\dlp = \frac{1}{2}+\trace\widetilde\dlp: H^1(\partial\Omega)\rightarrow H^1(\partial\Omega)$, from
which we conclude $\trace\widetilde\dlp v\in H^1(\Gamma)$.
Therefore,
\begin{align}
  \label{prop:invest:K:eq:0}
    \norm{\wght\nabla_\Gamma \dlp v}{L^2(\Gamma)}
     = \norm{\wght\nabla_\Gamma \big(\mbox{$\frac{1}{2}$}+\trace\widetilde\dlp\big) v}{L^2(\Gamma)}
    \leq \frac{1}{2}\norm{\wght\nabla_\Gamma v}{L^2(\Gamma)}
    + \norm{\wght\nabla_\Gamma \gamma_0^\interior u_\dlp}{L^2(\Gamma)}
\end{align}
with $u_\dlp = \widetilde\dlp v$.
There holds $u_\dlp +v_T = u_{\dlp,\el}^\near + u_{\dlp,\el}^\far$ in $\Omega$,
cf.~\eqref{eq:addition of near and far-field contributions of Ktilde}.
For the second term on the right-hand side in~\eqref{prop:invest:K:eq:0}, we obtain 
\begin{align}\label{prop:invest:K:eq:1}
\begin{split}
\norm{\wght\nabla_\Gamma \gamma_0^\interior u_\dlp }{L^2(\Gamma)}^2
& \leq \sum_{\el\in\mesh_\index}\norm{\wght/\meshsize^{1/2}}{L^\infty(T)}^2\norm{\meshsize^{1/2}\nabla_\Gamma \gamma_0^\interior (u_\dlp +v_T)}{L^2(\el)}^2 \\
& \overset{\eqref{eq:addition of near and far-field contributions of Ktilde}}{\lesssim} \sum_{\el\in\mesh_\index}\norm{\wght/\meshsize^{1/2}}{L^\infty(T)}^2\norm{\meshsize^{1/2}\nabla_\Gamma \gamma_0^\interior u_{\dlp,\el}^\near}{L^2(\el)}^2\\
&\qquad + \sum_{\el\in\mesh_\index}\norm{\wght/\meshsize^{1/2}}{L^\infty(T)}^2\norm{\meshsize^{1/2}\nabla_\Gamma \gamma_0^\interior u_{\dlp,\el}^\far}{L^2(\el)}^2.
\end{split}
\end{align}
The first sum can be bounded by Proposition~\ref{lemma:nearfield:K},
whereas the second sum can be bounded by Proposition~\ref{prop:farfield:K}.
Altogether, this yields
\begin{align*}
\norm{\wght\nabla_\Gamma \dlp v}{L^2(\Gamma)}
\lesssim \norm{\wght/\meshsize^{1/2}}{L^\infty(\Gamma)}\norm{v}{\H^{1/2}(\Gamma)} + \norm{\wght\nabla_\Gamma v}{L^2(\Gamma)}
\end{align*}
and concludes the first part of the proof.

The result for the hypersingular integral operator $\hyp$ is shown with similar arguments.
Let again $v \in \H^1(\Gamma)$ and $v_T$ as in Lemma~\ref{lemma:poincare}.
Note that $\hyp v_T=0$.
Splitting now into near-field and far-field yields
\begin{align}
%\begin{split}
\label{eq:invest:W:1}
\norm{\wght\hyp v}{L^2(\Gamma)}^2
&
= \sum_{T\in\mesh_\index}\norm{\wght\hyp(v-v_T)}{L^2(T)}^2
\nonumber
\\&
\lesssim \sum_{\el\in\mesh_\index} \norm{\wght\hyp((v-v_T)\eta_\el)}{L^2(\el)}^2
+ \sum_{\el\in\mesh_\index}\norm{\wght\hyp((v-v_T)(1-\eta_\el))}{L^2(\el)}^2.
%\end{split}
\end{align}
The near-field contribution is bounded by the stability of 
$\hyp:H^1(\partial\Omega)\rightarrow L^2(\partial\Omega)$
stated in \eqref{def:hyp} and 
the Poincar\'e-type estimate \eqref{est:poincare on patch - third est}
\begin{align*}
\norm{\hyp((v-v_T)\eta_\el)}{L^2(\el)}^2
\overset{\eqref{def:hyp}}{\lesssim} \norm{( v-v_T)\eta_\el}{H^1(\omega_\index(\el))}^2
\overset{\eqref{est:poincare on patch - third est}}{\lesssim} \norm{\nabla_\Gamma v}{L^2(\omega_\index(\el))}^2.
\end{align*}
The sum over all elements gives
\begin{align*}
\sum_{\el\in\mesh_\index} \norm{\wght\hyp((v-v_T)\eta_\el)}{L^2(\el)}^2
\lesssim \sum_{\el\in\mesh_\index} \norm{\wght}{L^\infty(T)}^2\norm{\nabla_\Gamma v}{L^2(\omega_\index(\el))}^2
\lesssim \norm{\wght\nabla_\Gamma v}{L^2(\Gamma)}^2.
  \end{align*}
It remains to bound the second term on the right-hand side in~\eqref{eq:invest:W:1}.
In view of the support properties of $\eta_T$, the potential 
$u_{\dlp,T}^\far = \widetilde \dlp ((v - v_T)(1-\eta_T))$
is smooth near $T$ (cf. Lemma~\ref{lemma:caccioppoli:K}) so that 
$\gint u_{\dlp,T}^\far = \partial_{\normal} u_{\dlp,T}^\far$ on $T$. 
Furthermore, since 
$\hyp=-\gamma_1^\interior\widetilde\dlp$ we see
\begin{align*}
\norm{\hyp((v-v_T)(1-\eta_\el))}{L^2(\el)}^2
= \norm{\gint u_{\dlp,\el}^\far}{L^2(\el)}^2
= \norm{\partial_\normal u_{\dlp,\el}^\far}{L^2(\el)}^2
\leq \norm{\nabla u_{\dlp,\el}^\far}{L^2(\el)}^2.
  \end{align*}
We use Proposition~\ref{prop:farfield:K} to conclude
\begin{align*}
\sum_{\el\in\mesh_\index}\norm{\wght\hyp((v-v_T)(1-\eta_\el))}{L^2(\el)}^2
&
\leq \sum_{\el\in\mesh_\index}\norm{\wght\nabla u_{\dlp,\el}^\far}{L^2(\el)}^2
\\ &
\lesssim \norm{\wght\nabla_\Gamma v}{L^2(\Gamma)}^2
+ \norm{\wght/\meshsize^{1/2}}{L^\infty(\Gamma)}^2\norm{v}{\H^{1/2}(\Gamma)}^2.
  \end{align*}
Altogether, we obtain
\begin{align*}
\norm{\wght\hyp v}{L^2(\Gamma)}
&\lesssim \norm{\wght\nabla_\Gamma v}{L^2(\Gamma)}
+ \norm{\wght/\meshsize^{1/2}}{L^\infty(\Gamma)}\norm{v}{\H^{1/2}(\Gamma)}. 
\qedhere
  \end{align*}
\end{proof}
%
%
% === Proof Of Corollary ===============================================================================

\bigskip
\noindent
\textbf{Acknowledgement.}
The authors are supported by the Austrian Science Fund (FWF) under grant
P27005 (MF, DP) as well as through the FWF doctoral school W1245 (MF, JMM, DP).
The are also supported by CONICYT through FONDECYT projects 3150012 (TF)
and 3140614 (MK).

%------------------------------------
\appendix
\section{}

\begin{lemma}
\label{lemma:hpinvest-beweis}
Let $\mesh_\index$ be a regular, $\kappa$-shape regular triangulation of $\Gamma$. 
Suppose that $d\ge 2$ and
that $q_\index$ is a $\sigma$-admissible polynomial degree 
distribution with respect to $\TT_\index$. 
Then, there exists a constant $\setc{invtilde}>0$, which depends solely on $\partial\Omega$, 
the $\kappa$-shape regularity of $\mesh_\index$, and $\sigma$, such that 
  \begin{align}
    \label{eq:foo-1000}
    \norm{\meshsize^{1/2}(q_\index+1)^{-1}%\max\{1,q_\index\}^{-1}
    \,\Psi_\index}{L^2(\Gamma)}
    &\leq \setc{invtilde} \norm{\Psi_\index}{H^{-1/2}(\Gamma)}
    \quad\text{for all }\Psi_\index\in\PP^\q(\mesh_\index). 
  \end{align}
\end{lemma}

\begin{proof} 
For each $T \in \mesh_\index$, let $\widehat \chi_{T,q_h(T)}$ be a smooth function on $\Tref$ 
with the following properties for some fixed $\delta > 0$
(see, e.g., the proofs of \cite[Lem.~{3.7}, Prop.~{3.8}]{georgoulis} or the arguments below):
\begin{align}
\label{eq:cor:invest-5}
&\operatorname*{supp} \widehat \chi_{T,q_h(T)} 
  \subseteq \{x \in \Tref\colon \operatorname*{dist}(x,\partial\Tref) > \delta /(q_h(T)+1)^2\}, \\
&0 \leq \widehat \chi_{T,q_h(T)} \leq 1 \quad \mbox{ in $\Tref$}, 
\qquad \|\nabla \widehat \chi_{T,q_h(T)}\|_{L^\infty(\Tref)} \lesssim (q_h(T)+1)^{-2},  \\
&\widehat \chi_{T,q_h(T)} \equiv 1 \quad \mbox{ in }
  \{x \in \Tref\colon \operatorname*{dist}(x,\partial\Tref) > 3\delta /(q_h(T)+1)^2\}, \\
\label{eq:cor:invest-10}
&\|\pi\|_{L^2(\Tref)} \leq C \|\pi \widehat \chi_{T,p(T)}\|_{L^2(\Tref)} 
\quad \mbox{$\forall$ polynomials $\pi$ of degree $q_h(T)$}, \\ 
\label{eq:cor:invest-20}
&\|\pi \widehat \chi_{T,q_h(T)}\|_{H^1(\Tref)} \leq C (1+q_h(T))^{2} \|\pi\|_{L^2(\Tref)}
\quad \mbox{$\forall$ polynomials $\pi$ of degree $q_h(T)$}.  
\end{align}
$\widehat \chi_{T,q_h(T)}$ is obtained from a mollification of the characteristic function of 
$\Tref \setminus S_{2\delta/(q_h(T)+1)^2}$, where 
$S_\varepsilon:= \{x \in \Tref \colon \operatorname*{dist}(x,\partial\Tref) < \varepsilon\}$. 
The parameter $\delta > 0$ is dictated by the requirement \eqref{eq:cor:invest-10}. 
For this, we use the shorthand $\varepsilon(\delta) = 3 \delta /(q_h(T)+1)^2$ and observe that 
we assume $\widehat \chi_{T,q_h(T)} \equiv 1$ on $\Tref\setminus S_{\varepsilon(\delta)}$
so that we are done once we have established 
$\|\pi\|_{L^2(S_{\varepsilon(\delta)})} \lesssim \|\pi\|_{L^2(\Tref\setminus S_{\varepsilon(\delta)})}$. 
\cite[Lemma~{2.1}]{li-melenk-wohlmuth-zou10} and the polynomial inverse 
estimate $\|\pi\|_{H^1(\Tref)} \lesssim (q_h(T)+1)^2 \|\pi\|_{L^2(\Tref)}$, yield 
\begin{align*}
\|\pi\|^2_{L^2(S_{\varepsilon(\delta)})} &\lesssim 
\varepsilon(\delta) \|\pi\|_{L^2(\Tref)} \|\pi\|_{H^1(\Tref)}
\lesssim \varepsilon(\delta) (q_h(T)+1)^2 \|\pi\|^2_{L^2(\Tref)}   \\
&= \varepsilon(\delta) (q_h(T)+1)^2 \left[ \|\pi\|^2_{L^2(\Tref\setminus S_{\varepsilon(\delta)})} 
+ \|\pi\|^2_{L^2(S_{\varepsilon(\delta)})}\right]. 
\end{align*}
Taking $\delta$ sufficiently small produces 
$\|\pi\|_{L^2(S_{\varepsilon(\delta)})} \lesssim \|\pi\|_{L^2(\Tref\setminus S_{\varepsilon(\delta)})}$ as desired.
\newline 
Define $\chi_{T,q_h(T)}$ with $\operatorname*{supp} \chi_{T,q_h(T)} \subseteq T$ by 
$\chi_{T,q_h(T)} \circ \gamma_T = \widehat \chi_{T,q_h(T)}$. 
Given $\Psi_\index\in\PP^\q(\mesh_\index)$, define $\H^1(\Gamma) \ni v 
:= \sum_{T \in \mesh_\index} v_T$ in an elementwise fashion by requiring $\operatorname*{supp} v_T \subseteq T$ and 
\begin{equation}
\label{eq:cor:invest-30}
v_T|_T := \frac{h(T)}{(1+q_h(T))^2}(\Psi_\index|_{T}) \chi_{T,q_h(T)}, 
\end{equation}
Note that $v_T \in \H^1(\Gamma)$ by the support properties
of $\chi_{T,q_h(T)}$. 
An interpolation inequality and the estimate \eqref{eq:cor:invest-20} on the reference element give 
\begin{align*}
\|v_T \|^2_{\H^{1/2}(\Gamma)} &\leq
\|v_T \|^2_{\H^{1/2}(\partial\Omega)} 
\lesssim \|v_T\|_{L^2(\partial\Omega)} \|v_T\|_{H^1(\partial\Omega)}
 = \|v_T\|_{L^2(T)} \|v_T\|_{H^1(T)}  \\
&\stackrel{\eqref{eq:cor:invest-20}, \eqref{eq:cor:invest-10}}{\lesssim} \frac{(1+q_h(T))^2}{h(T)}\|v_T\|^2_{L^2(T)}. 
\end{align*}
This implies 
\begin{equation}
\label{eq:cor:invest-50}
\sum_{T \in \mesh_\index} \|v_T\|^2_{\H^{1/2}(\Gamma)} \lesssim \left\| \frac{q_h+1}{h^{1/2}} v\right \|^2_{L^2(\Gamma)}.
\end{equation}
From \cite[Lemma~{4.1.49}]{ss} we get from \eqref{eq:cor:invest-50}
\begin{equation}
\label{eq:cor:invest-60} 
\|v\|^2_{\H^{1/2}(\Gamma)} = \left\|\sum_{T \in \mesh_\index} v_T\right\|^2_{\H^{1/2}(\Gamma)} 
\lesssim \sum_{T \in \mesh_\index} \|v_T\|^2_{\H^{1/2}(\Gamma)}  
\lesssim \left\| \frac{1+q_h}{h^{1/2}} v\right\|^2_{L^2(\Gamma)}. 
\end{equation}
Finally, we estimate 
\begin{align*}
\left\|\frac{h^{1/2}}{1+q_h} \Psi_h\right\|^2_{L^2(\Gamma)} &= 
\sum_{T \in \mesh_\index} 
\left\|\frac{h(T)^{1/2}}{1+q_h(T)} \Psi_h\right\|^2_{L^2(T)} 
\stackrel{\eqref{eq:cor:invest-10}}{\lesssim}
\sum_{T \in \mesh_\index} 
\left\|\frac{h(T)^{1/2}}{1+q_h(T)} \chi_{T,q_h(T)} \Psi_h\right\|^2_{L^2(T)} \\
&
\stackrel{\eqref{eq:cor:invest-5}}{\lesssim}
\sum_{T \in \mesh_\index} 
\|\frac{h(T)^{1/2}}{1+q_h(T)} \sqrt{\chi_{T,q_h(T)}} \Psi_h\|^2_{L^2(T)} 
 = 
\sum_{T \in \mesh_\index} (v_T,\Psi_h)_{L^2(T)}
 = (v,\Psi_h)_{L^2(\Gamma)} \\
& \leq \|\Psi_h\|_{H^{-1/2}(\Gamma)} \|v\|_{\H^{1/2}(\Gamma)}
\stackrel{\eqref{eq:cor:invest-60} }{\lesssim}
\|\Psi_h\|_{H^{-1/2}(\Gamma)} \left\|\frac{1+q_h}{h^{1/2}} v\right\|_{L^{2}(\Gamma)} \\
&
\stackrel {\eqref{eq:cor:invest-30}} {\lesssim}
\|\Psi_h\|_{H^{-1/2}(\Gamma)} \left\|\frac{h^{1/2}}{1+q_h} \Psi_h\right\|_{L^{2}(\Gamma)}. 
\qedhere
\end{align*}
%which implies the desired estimate. 
\end{proof}
%------------------------------------
\bibliographystyle{alpha}
\bibliography{literature}

\newcommand{\etalchar}[1]{$^{#1}$}
\begin{thebibliography}{AFF{\etalchar{+}}13b}

\bibitem[AFF{\etalchar{+}}12]{invest}
Markus Aurada, Michael Feischl, Thomas F{\"u}hrer, Michael Karkulik,
  Jens~Markus Melenk, and Dirk Praetorius.
\newblock Inverse estimates for elliptic boundary integral operators and their
  application to the adaptive coupling of {FEM} and {BEM}.
\newblock {\em ASC Report}, 07/2012, {\em Institute for Analysis and Scientific
  Computing, Vienna University of Technology} and {\tt arXiv} 1211.4360, 2012.

\bibitem[AFF{\etalchar{+}}13a]{fembem}
Markus Aurada, Michael Feischl, Thomas F{\"u}hrer, Michael Karkulik,
  Jens~Markus Melenk, and Dirk Praetorius.
\newblock Classical {FEM}-{BEM} coupling methods: {N}onlinearities,
  well-posedness, and adaptivity.
\newblock {\em Comput. Mech.}, 51(4):399--419, 2013.

\bibitem[AFF{\etalchar{+}}13b]{eps65}
Markus Aurada, Michael Feischl, Thomas F{\"u}hrer, Michael Karkulik, and Dirk
  Praetorius.
\newblock Efficiency and optimality of some weighted-residual error estimator
  for adaptive 2{D} boundary element methods.
\newblock {\em Comput. Methods Appl. Math.}, 13(3):305--332, 2013.

\bibitem[AFF{\etalchar{+}}14]{hypsing}
Markus Aurada, Michael Feischl, Thomas F{\"u}hrer, Michael Karkulik, and Dirk
  Praetorius.
\newblock Energy norm based error estimators for adaptive {BEM} for
  hypersingular integral equations.
\newblock {\em Applied Numerical Mathematics}, (0):--, 2014.

\bibitem[Car97]{cc1997}
Carsten Carstensen.
\newblock An a posteriori error estimate for a first-kind integral equation.
\newblock {\em Math. Comp.}, 66(217):139--155, 1997.

\bibitem[CKNS08]{ckns}
J.~Manuel Cascon, Christian Kreuzer, Ricardo~H. Nochetto, and Kunibert~G.
  Siebert.
\newblock Quasi-optimal convergence rate for an adaptive finite element method.
\newblock {\em SIAM J. Numer. Anal.}, 46(5):2524--2550, 2008.

\bibitem[CMPS04]{cmps}
Carsten Carstensen, Matthias Maischak, Dirk Praetorius, and Ernst~P. Stephan.
\newblock Residual-based a posteriori error estimate for hypersingular equation
  on surfaces.
\newblock {\em Numer. Math.}, 97(3):397--425, 2004.

\bibitem[CMS01]{cms}
Carsten Carstensen, Matthias Maischak, and Ernst~P. Stephan.
\newblock A posteriori error estimate and {$h$}-adaptive algorithm on surfaces
  for {S}ymm's integral equation.
\newblock {\em Numer. Math.}, 90(2):197--213, 2001.

\bibitem[CP06]{ccdpr:symm}
Carsten Carstensen and Dirk Praetorius.
\newblock Averaging techniques for the effective numerical solution of {S}ymm's
  integral equation of the first kind.
\newblock {\em SIAM J. Sci. Comput.}, 27(4):1226--1260, 2006.

\bibitem[DFG{\etalchar{+}}04]{dfghs}
W.~Dahmen, B.~Faermann, I.~G. Graham, W.~Hackbusch, and S.~A. Sauter.
\newblock Inverse inequalities on non-quasi-uniform meshes and application to
  the mortar element method.
\newblock {\em Math. Comp.}, 73(247):1107--1138, 2004.

\bibitem[DS80]{ds80}
Todd Dupont and Ridgway Scott.
\newblock Polynomial approximation of functions in {S}obolev spaces.
\newblock {\em Math. Comp.}, 34(150):441--463, 1980.

\bibitem[EG92]{eg92}
L.C. Evans and R.F. Gariepy.
\newblock {\em Measure Theory and Fine Properties of Functions}.
\newblock CRC Press, 1992.

\bibitem[FFK{\etalchar{+}}14]{part1}
Michael Feischl, Thomas F{\"u}hrer, Michael Karkulik, Jens~Markus Melenk, and
  Dirk Praetorius.
\newblock Quasi-optimal convergence rates for adaptive boundary element methods
  with data approximation, part {I}: weakly-singular integral equation.
\newblock {\em Calcolo}, 51(4):531--562, 2014.

\bibitem[FFK{\etalchar{+}}15]{part2}
Michael Feischl, Thomas F\"uhrer, Michael Karkulik, {Jens Markus} Melenk, and
  Dirk Praetorius.
\newblock Quasi-optimal convergence rates for adaptive boundary element methods
  with data approximation, part {II}: {H}yper-singular integral equation.
\newblock {\em Electron. Trans. Numer. Anal.}, 44:153--176, 2015.

\bibitem[FKMP13]{fkmp}
Michael Feischl, Michael Karkulik, {J.\ Markus} Melenk, and Dirk Praetorius.
\newblock Quasi-optimal convergence rate for an adaptive boundary element
  method.
\newblock {\em SIAM J. Numer. Anal.}, 51:1327--1348, 2013.

\bibitem[Gan13]{gantumur}
Tsogtgerel Gantumur.
\newblock Adaptive boundary element methods with convergence rates.
\newblock {\em Numer. Math.}, 124(3):471--516, 2013.

\bibitem[Geo08]{georgoulis}
Emmanuil~H. Georgoulis.
\newblock Inverse-type estimates on {$hp$}-finite element spaces and
  applications.
\newblock {\em Math. Comp.}, 77(261):201--219, 2008.

\bibitem[GHS05]{ghs}
Ivan~G. Graham, Wolfgang Hackbusch, and Stefan~A. Sauter.
\newblock Finite elements on degenerate meshes: inverse-type inequalities and
  applications.
\newblock {\em IMA J. Numer. Anal.}, 25(2):379--407, 2005.

\bibitem[HW08]{hsiaowendland}
George~C. Hsiao and Wolfgang~L. Wendland.
\newblock {\em Boundary integral equations}, volume 164 of {\em Applied
  Mathematical Sciences}.
\newblock Springer-Verlag, Berlin, 2008.

\bibitem[KM14]{km14}
Michael Karkulik and Jens~Markus Melenk.
\newblock Local high-order regularization and applications to hp-methods
  (extended version).
\newblock Technical Report 1411.5209, {\tt arXiv}, 2014.

\bibitem[KMR14]{kmr14}
Michael Karkulik, {Jens Markus} Melenk, and Alexander Rieder.
\newblock Optimal additive schwarz methods for the $p$-{BEM}: the hypersingular
  integral operator.
\newblock {\em work in progress}, 2014.

\bibitem[LMWZ10]{li-melenk-wohlmuth-zou10}
Jingzhi Li, Jens~Markus Melenk, Barbara Wohlmuth, and Jun Zou.
\newblock Optimal a priori estimates for higher order finite elements for
  elliptic interface problems.
\newblock {\em Appl. Numer. Math.}, 60(1-2):19--37, 2010.

\bibitem[McL00]{mclean}
William McLean.
\newblock {\em Strongly elliptic systems and boundary integral equations}.
\newblock Cambridge University Press, Cambridge, 2000.

\bibitem[Mor08]{morrey66}
Charles~B. Morrey, Jr.
\newblock {\em Multiple integrals in the calculus of variations}.
\newblock Classics in Mathematics. Springer-Verlag, Berlin, 2008.
\newblock Reprint of the 1966 edition.

\bibitem[SS11]{ss}
Stefan~A. Sauter and Christoph Schwab.
\newblock {\em Boundary element methods}, volume~39 of {\em Springer Series in
  Computational Mathematics}.
\newblock Springer-Verlag, Berlin, 2011.
\newblock Translated and expanded from the 2004 German original.

\bibitem[Ste70]{stein70}
E.M. Stein.
\newblock {\em Singular integrals and differentiability properties of
  functions}.
\newblock Princeton University Press, 1970.

\bibitem[Ste07]{stevenson}
Rob Stevenson.
\newblock Optimality of a standard adaptive finite element method.
\newblock {\em Found. Comput. Math.}, 7(2):245--269, 2007.

\bibitem[SZ90]{scottzhang}
L.~Ridgway Scott and Shangyou Zhang.
\newblock Finite element interpolation of nonsmooth functions satisfying
  boundary conditions.
\newblock {\em Math. Comp.}, 54(190):483--493, 1990.

\bibitem[Ver84]{verchota}
Gregory Verchota.
\newblock Layer potentials and regularity for the {D}irichlet problem for
  {L}aplace's equation in {L}ipschitz domains.
\newblock {\em J. Funct. Anal.}, 59(3):572--611, 1984.

\end{thebibliography}

\end{document}